\documentclass[hidelinks,onefignum,onetabnum]{siamart251216}

\usepackage{color,comment}
\usepackage{amsfonts}
\usepackage{graphicx}
\usepackage{epstopdf}
\usepackage{algorithmic}
\ifpdf
  \DeclareGraphicsExtensions{.eps,.pdf,.png,.jpg}
\else
  \DeclareGraphicsExtensions{.eps}
\fi

\newsiamremark{remark}{Remark}
\newcommand{\uld}[1]{\underline{d#1}}

\newcommand\nlc{\eta}
\newcommand\obs{p}
\newcommand\uvec{\vec{u}}
\newcommand\obsvec{\vec{\obs}}
\newcommand\rvec{\vec{r}}

\newcommand\N{\mathbb{N}}
\newcommand\XXX{\mathbb{X}}
\newcommand\YYY{\mathbb{Y}}
\newcommand\pole{\mathfrak{p}}

\newcommand\no{n}
\newcommand\No{N}
\newcommand\orti{\mathfrak{s}}
\newcommand\slo{\sigma}
\newcommand\Chat{\hat{C}}
\newcommand\Ccheck{\check{C}}

\newcommand\tint{{\textstyle \int\!}_t\,}
\newcommand\Rho{G}
\newcommand\Tinv{\mathcal{T}_\dagger}
\renewcommand\Re{\mathrm{Re}}

\definecolor{darkcyan}{rgb}{0.,0.5,0.5}
\definecolor{darkviolet}{rgb}{0.5,0.,0.5}
\definecolor{darkgreen}{rgb}{0.,0.7,0.}
\definecolor{purple}{rgb}{0.7,0.,0.7}

\newcommand{\marginrr}[1]{}
\newcommand{\marginR}[1]{}

\newcommand{\marginre}[1]{}

\headers{Nonlinearity and sound speed imaging with the JMGT equation}{B. Kaltenbacher}

\title{Imaging nonlinearity coefficient and sound speed with the JMGT equation in frequency domain\thanks{Submitted to the editors DATE.
\funding{This research was funded in part by the Austrian Science Fund (FWF) 
[10.55776/P36318].}}}

\author{Barbara Kaltenbacher\thanks{University of Klagenfurt, Austria 
  (\email{barbara.kaltenbacher@aau.at}, \url{https://me.aau.at/~bkaltenb/}).}
}

\ifpdf
\hypersetup{
  pdftitle={An Example Article},
  pdfauthor={D. Doe, P. T. Frank, and J. E. Smith}
}
\fi

\begin{document}

\maketitle

% REQUIRED
\begin{abstract}
In this paper we prove uniqueness and stability of reconstruction of two coefficients (sound speed and nonlinearity parameter) in the Jordan-Moore-Gibson-Thompson JMGT equation of nonlinear acoustics, relying on observations resulting from only two sources.
A key tool for this purpose is a multiharmonic expansion of the PDE solution, which reflects the physical phenomenon of higher harmonics appearing due to nonlinearity and allows us to work in frequency domain. 
Based on this result, we derive a regularization property of reconstruction with JMGT as the relexation time tends to zero (in the spirit of a quasi reversibility method) for reconstruction from the classical Westervelt equation.
\end{abstract}

% REQUIRED

% REQUIRED

\section{Introduction}
Taking into account and even exploiting nonlinear wave propagation in quantitative ultrasound imaging is a topic that has recently found much interest, due to its potential of increasing diagnostic value, while reducing measurement effort.
 
Acoustic nonlinearity parameter tomography ANT~\cite{nonlinparam1, duck2002nonlinear,  nonlinparam3, nonlinparam2, ZHANG20011359} relies on the tissue specific and thus spatially varying nonlinearity coefficient as an imaging quantity, while ultrasound tomography UST~\cite{Lucka2022,schmitt2002ultrasound} reconstructs linear tissue properties such as the sound speed or the attenuation as a function of space.  
We will here study a combination of both.

The underlying mathematical analysis involves partial differential equation PDE models in which these imaging quantities appear as coefficients.
A classical model of nonlinear acoustics is the Westervelt equation, a second order in time quasilinear strongly damped wave equation that due to the nonlinearity exhibits potential nonlinearity. 
We refer to, e.g., \cite{%CelikKyed:2018,CelikKyed:2019, THIS IS KUZNETSOV ANF BLACKSTOCK_CRIGHTON
periodicWestervelt,periodicWest_2} for results on its well-posedness in the time-periodic setting relevant here and to ~\cite{AcostaUhlmannZhai2021,AcostaPalacios2025,Eptaminitakis:Stefanov:2023,nonlinearity_imaging_Westervelt,nonlinearity_imaging_fracWest,nonlinearity_imaging_2d}
for recent work on the inverse problem of ANT with the Westervelt equation.
Since this conventional model exhibits the infinite speed of propagation paradox, advanced models have been studied. Among these is the Jordan Moore Gibson Thompson JMGT equation, that by employing the Maxwell-Cattaneo law of heat conduction (rather than the Fourier one), thus adding a relaxation time term, amends this unphysical effect. Well-posedness of the time-periodic JMGT equation has been studied in \cite{nonlinear_imaging_JMGT_freq} and the inverse problem of ANT with the JMGT equation has been investigated in, e.g., ~\cite{nonlinearity_imaging_JMGT,nonlinear_imaging_JMGT_freq}, where ANT has been formulated in time and frequency domain, for the latter purpose using a multiharmonic expansion of the PDE solution 
\cite{BachingerLangerSchoeberl:2005,periodicWestervelt,periodicWest_2,nonlinear_imaging_JMGT_freq}
There it was shown that nonlinear interaction of ultrasound waves renders data from a single source rich enough to uniquely determine the nonlinearity coefficient as a function of two or three dimensional space. This is in stark contrast to the linear wave setting, where the analog reconstruction task (for, e.g., the sound speed) requires observations resulting from infinitely many sources, namely, the whole Neumann-to-Dirichlet map. 

A key message of this paper is that this ``blessing of nonlinearity'' \cite{KurylevLassasUhlmann2018} enables reconstruction of not only the nonlinearity coefficient but also of linear parameters such as the sound speed from measurements generated by few (more precisely, two) sources. 

Our second main contribution is a substantiation of the intuitive insight that introducing a relaxation time term not only leads to a physically more meaningful model, but also remedies the severe loss of information due to strong damping present in the Westervelt equation, thus allowing us to use the relaxation time as a regularization parameter for the inverse problem, in the sense of a quasi reversibility regularization method~\cite{AmesClark_etal:1998,ClarkOppenheimer:1994,LattesLions:1969,Showalter:1974a,Showalter:1974b,Showalter:1976} and \cite[Section 8.3]{frac_book}.

The analysis is here carried out in frequency domain, as this is not only relevant in practical settings of continuous wave excitation, but also allows to formulate the arguments in a very transparent way. In this setting, the multiplication of information through nonlinearity naturally shows up by the apperance of contributions at multiples of the fundamental frequency, so-called higher harmonics.

\medskip

We consider nonlinear ultrasound propagation as described by the JMGT equation
\begin{equation}\label{JMGT-Westervelt}
\tau u_{ttt}+\frac{1}{c^2} u_{tt}-\Delta u - \beta \Delta u_t - \nlc (u^2)_{tt} +f=0 \ \text{ in }(0,T)\times\Omega
\end{equation}
with time periodicity conditions
\begin{equation}\label{periodic}
%\int_0^T u(t)\, dt =0, \quad 
u(T)=u(0), \quad u_t(T)=u_t(0)
, \quad \tau u_{tt}(T)=\tau u_{tt}(0)
\end{equation}
where the sound speed $c=c(x)$ and the nonlinearity coefficient $\nlc=\nlc(x)$ depend on the tissue type, thus vary in space and can be used as imaging quantities, while the relaxation time $\tau>0$ and the attenuation coefficient $\beta>0$ are considered as constants.
Typically $\beta\geq\tau c^2$ holds, which is in fact a requirement for stability of the initial value probem (instead of the periodicity one \eqref{periodic}) for the PDE \eqref{JMGT-Westervelt}, cf., e.g., \cite{BongartiCharoenphonLasiecka20,DellOroPata,KLP12_JordanMooreGibson,JMGT,LasieckaWang15b,MarchandMcDevittTriggiani12,PellicerSolaMorales} .
The above PDE \eqref{JMGT-Westervelt} is supposed to hold on a smooth bounded domain $\Omega\subseteq\mathbb{R}^d$, $d\in\{2,3\}$ and complemented with boundary conditions
\begin{equation}\label{bndy}
\begin{aligned}
&\partial_\nu u
%+\beta u_t
+\gamma u =0 \mbox{ on }%\Gamma_a\cup
\Gamma_i\cup\Gamma_N\\
&u =0 \mbox{ on }\Gamma_D\\
&\text{ where $\gamma=0$ on $\Gamma_N$, $\tfrac{1}{\gamma}\vert_{\Gamma_i}\in L^\infty(\Gamma_i)$, 
%$\beta=0$ on $\Gamma_i\cup\Gamma_N$,
}
\quad \overline{\Gamma_i}\cup\overline{\Gamma_N}\cup\overline{\Gamma_D}=\partial\Omega.
\end{aligned}
\end{equation}

The quantitative imaging task under consideration amounts to reconstructing the space dependent coefficients $c(x)$, $\nlc(x)$ in \eqref{JMGT-Westervelt}, \eqref{periodic}, \eqref{bndy} from 
$\No\geq2$ 
observations of the acoustic pressure 
\begin{equation}\label{observation}
\obs_\no(t,x_0) = u(t,x_0), \quad(t,x_0)\in(0,T)\times \Sigma, \quad \no\in\{1,\ldots,\No\},
\end{equation}
where $\Sigma\subseteq\overline{\Omega}$ is s smooth $d-1$ dimensional manifold, modeling, e.g. an array of piezoelectric transducers or hydrophones.
For simplicity of exposition, we remain with $\No=2$ here; indeed we will show that two observations suffice to uniquely recover the two coefficient functions $c(x)$, $\nlc(x)$.

The two sources $f_1$, $f_2$ that we use for this purpose can be excited from a lower dimensional manifold $\Gamma\subseteq\overline{\Omega}$, e.g., some boundary part, as practically relevant, cf. Remark~\ref{rem:u0separable}, using the known physical effect of nonlinear wave interaction giving rise to interior sources. 
By simply setting $f_2$ to be a multiple of $f_1$, we exploit another aspect of nonlinearity, namely the fact that amplitude modulation yields additional information (as opposed to the linear superposition principle).

The limiting case $\tau=0$ in \eqref{JMGT-Westervelt} corresponds to the classical Westervelt equation. While being a widely used model, it has the following two drawbacks: First of all, it exhibits the above mentioned infinite speed of propagation paradox; secondly (and related to the first), the severe ill-posedness of the inverse problem, as a consequence of information loss due to strong attenuation.

The main outcomes of this paper are as follows.
\begin{itemize}
\item to prove local uniqueness and conditional stability of $c^2$ and $\nlc$ as functions of space simultaneously from two boundary observations, see Section~\ref{sec:uniqueness-nl};
\item to establish $(c^2,\nlc)$ imaging with JMGT as a regularization method with $\tau$ as regularization parameter, in the spirit of quasi reversibility, see Section~\ref{sec:quasireversibility}. 
\end{itemize}
The local uniqueness result relies on quantified linearized stability and uniqueness in appropriately chosen norms, see Section~\ref{sec:uniqueness-lin}.

\begin{remark}[recovery of further space-dependent coefficients]\label{rem:complexwavespeed_rho}
Considering $u$ as a complex valued quantity, thereby combining $c(x)$ and $\beta(x)$ into a single but complex valued ``lossy wave speed'', one could additionally recover the attenuation $\beta$ as a function of space; we point to \cite{AcostaPalacios2025}, where the observation setting is different from here, though: In \cite{AcostaPalacios2025}, the whole Robin-to-Dirichlet map for the first and second harmonic is considered, whereas here we only use two boundary observations and all higher harmonics (as realized by observations over the time interval).
Alternatively, sources at different fundamental frequencies can be used to uniquenly recover $c(x)$, $\beta(x)$ and $\nlc(x)$ simultaneously, \cite{periodicWest_3}.

To also take into account spatial variations of the mass density $\rho$, we could as well use $\Delta_\rho=\rho\nabla\cdot(\frac{1}{\rho}\nabla\,\cdot)$ in place of $\Delta$.
Recovery of information on $\rho(x)$ could rely on recovery of the eigenvalues of $\Delta_\rho$ 
(which can be obtained from locations of the poles $\pole_\ell$, see \eqref{poles_asymp} below) 
and application of inverse spectral results; note however, that the eigenvalue sequence alone does not determine the coefficient uniquely, see, e.g., \cite{ChadanColtonPaivarintaRundell:1997,KatchalovKurylevLassas:2001} and the references therein. 
Also the particular form $\rho\nabla\cdot(\frac{1}{\rho}\nabla\,\cdot)$ of the operator reveals the fact that  uniqueness of $\rho$ can only hold up to a constant scaling factor.
\end{remark}

Introducing the squared slowness
\[
\slo(x)=c^{-2}(x)
\]
where $\slo(x)\beta\geq\tau$ for almost every $x\in\Omega$ and 
the Robin Laplacian 
\begin{equation}\label{eqn:calADM}
\begin{aligned}
&\mathcal{A}u = \Bigl(v\mapsto \int_\Omega\nabla u\cdot\nabla v\, dx +\gamma\int_{\partial\Omega} u\, v\, ds\Bigr),
\end{aligned}
\end{equation}
we can write the inverse problem %(in time domain) 
as 
\\
%\begin{quote}
Reconstruct $\slo(x)$, $\nlc(x)$ in 
\begin{equation}\label{IP}
\begin{aligned}
&\text{model:}&&{}
[\tau\partial_t^3+\slo\partial_t^2 + \mathcal{A} + \beta \mathcal{A}\partial_t]u_\no - \nlc\, (u_\no^2)_{tt} = r_{\no,tt}\\
&&&{} 
%\int_0^T u_\no(t)\, dt =0, \quad 
u_\no(T)=u_\no(0), \quad u_{\no,t}(T)=u_{\no,t}(0), \quad 
\tau u_{\no,tt}(T)=\tau u_{\no,tt}(0)
\\
&\text{observation:}&& u_\no(x_0,t)=\obs_\no(x_0,t) \quad x_0\in\Sigma,
\\
&\no\in\{1,2\}.
\end{aligned}
\end{equation}
%\end{quote}

Note that the excitations naturally (to match physical units, cf. \cite{periodicWest_2}) appear as the second time derivative of pressure sources $r_\no$; this also equips them with sufficient strength in terms of the order in which they appear in the PDE.

%Moreover the requirements we will make to prove uniqueness allow to concentrate the sources $r_\no$ to a lower dimensional manifold, thus allowing to model excitation by an array of piezoelectric transducers, cf. Remark~\ref{rem:u0separable}.

\medskip

Relying on existence of periodic solutions \cite{periodicJMGT}, for a (multi-)harmonic excitation
$r_\no(x,t)=\Re\left(\sum_{k=1}^\infty \hat{r}_k(x) e^{\imath k \omega t}\right)$ (with $\omega=\frac{2\pi}{T}$), 
we can write the solution to \eqref{JMGT-Westervelt} via a multiharmonic ansatz
\begin{equation}\label{u_multiharmonic}
u_\no(x,t)
= \Re\left(\sum_{k=1}^\infty \hat{u}_{\no,k}(x) e^{\imath k \omega t}\right),
\end{equation}
Thus, we definine the forward operator and data by
\[
F_{\no,m}:= \left(\begin{array}{l}
F_{\no,m}^{mod} \\
F_{\no,m}^{obs}
\end{array}\right), \qquad
y_{\no,m}:= \left(\begin{array}{l}
\hat{r}_{\no,m} \\
\hat{\obs}_{\no,m}
\end{array}\right),
\]
with 
\begin{subequations}\label{IPfreq}
\begin{alignat}{2}
&%\text{model:}\quad 
F_{\no,m}^{mod}(\slo,\nlc,\uvec_\no):=
\mathcal{L}_m(\slo)\hat{u}_{\no,m} +\nlc\, \mathcal{B}_m(\uvec_\no,\uvec_\no) 
%&&= \hat{r}_{\no,m}
\label{IPfreq_mod}\\
&%\text{observation:}\quad 
F_{\no,m}^{obs}(\slo,\nlc,\uvec_\no):= 
\text{tr}_\Sigma\hat{u}_{\no,m} 
%&&=\obs_{\no,m},
\label{IPfreq_obs}
\end{alignat}
\end{subequations}
for all $m\in\mathbb{N}$, $\no\in\{1,2\}$, where $\text{tr}_\Sigma$ denotes the trace operator, applied pointwise almost everywhere in time, and we abbreviate the multiplication operator $v\mapsto \nlc\,v$ by just writing the multiplier $\nlc$. 
Moreover,
\begin{equation}\label{Bm}
\begin{aligned}
&\uvec=(\hat{u}_m)_{m\in\mathbb{N}}, \quad \hat{u}_m(x)={\frac{2}{T}}\int_0^T  u(x,t)\,e^{-\imath m \omega t}\, dt, \ x\in\Omega,\\
&\mathcal{L}_m(\slo) = \frac{1}{m^2\omega^2}
[\imath m^3\omega^3\tau+m^2\omega^2\slo
{-} \mathcal{A} {-} \beta \imath m\omega \mathcal{A}]\\
&\mathcal{B}_m(\uvec,\vec{v}) = 
\frac12 \sum_{\ell=1}^{m-1} \hat{u}_\ell \hat{v}_{m-\ell}
+\frac12 \sum_{k=1}^\infty\overline{\hat{u}_{k}} \hat{v}_{k+m} 
+\frac12 \sum_{k=1}^\infty\hat{u}_{m+k} \overline{\hat{v}_{k}}
%={\frac{2}{T}}\int_0^T  {u(\cdot,t)\,v(\cdot,t)}\,e^{-\imath m \omega t}\, dt,
\\
&\obsvec=(\hat{\obs}_m)_{m\in\mathbb{N}}, \quad 
\hat{\obs}_m(x_0)={\frac{2}{T}}\int_0^T  \obs(x_0,t)\,e^{-\imath m \omega t}\, dt , \ x_0\in\Sigma.
\end{aligned}
\end{equation}
With this, we can write the inverse problem as an operator equation in frequency domain.
\begin{equation}\label{Fnlcuvecy}
\vec{F}(\slo,\nlc,(\uvec_\no)_{\no=1}^2) = \vec{y}.
\end{equation}

The corresponding time domain formulation is  
\begin{equation}\label{Fnlcuvecy_timedom}
F(\slo,\nlc,(u_\no)_{\no=1}^2) = y
\end{equation}
where 
\[
F_{\no}:= \left(\begin{array}{l}
F_{\no}^{mod} \\
F_{\no}^{obs}
\end{array}\right), \qquad
y_{\no}:= \left(\begin{array}{l}
r_{\no} \\
\obs_{\no}
\end{array}\right)
\]
\begin{equation}\label{IPtime}
\begin{aligned}
&F_\no^{mod}(\slo,\nlc,u_\no):=
\mathcal{L}(\slo)u_{\no} +\nlc\, u_\no^2 \\
&F_\no^{obs}(\slo,\nlc,u_\no):= 
\text{tr}_\Sigma u_\no ,\\
&\mathcal{L}(\slo) = -\tau \partial_t - \slo - (\beta\tint +\tint\tint)\mathcal{A},
\end{aligned}
\end{equation}
where $(\tint v)(t)=\int_0^t v(t')\, dt'$ and we equip $\mathcal{L}(\slo)$ with periodicity conditions 
%\[
%(\tint^2 u)(T)=0, \quad (\tint u)(T)=0, \quad 
$u(T)=u(0)$.
%\]
%which results from considering it on the closure in $H^1$ of $\text{span}\{t\mapsto\sin(m\omega t),\ t\mapsto\cos(m\omega t)\,:\, m\in\mathbb{N}_0\}$, noting that $m=0$ is included in \cite{periodicWestervelt}.

\medskip

\subsubsection*{Some Notation}
By compactness and symmetry arguments (e.g., 
\cite{EvansBook})
%\cite{nonlinear_imaging_JMGT_freq}) 
one can show existence of an eigensystem $(\lambda^j,(\varphi^{j,k})_{k\in K^j})_{j\in\N}$ with $\lambda^j\nearrow\infty$ as $j\to\infty$ of the operator $\mathcal{A}$ that allows us to diagonalize it and to define a norm that is equivalent to the $H^s(\Omega)$ Sobolev norm as well as the eigenspaces $\mathbb{E}^j$
\begin{equation}\label{eigensystem}
\begin{aligned}
&\mathcal{A}\varphi^{j,k}=\lambda^j\varphi^{j,k}, \quad \langle \varphi^{j,k},\varphi^{i,\ell}\rangle=\delta_{ji}\delta_{k\ell}, \\
&\|v\|_{H^s(\Omega)}=\left(\sum_{j\in\N} (\lambda^j)^s\sum_{k\in K^j} |\langle \varphi^{j,k},v\rangle|^2\right)^{1/2},\\
&\mathbb{E}^j=\text{span}\{\varphi^{j,k}\,:\, k\in K^j\},
\end{aligned}
\end{equation}
where $\langle \cdot,\cdot\rangle$ denotes the $L^2(\Omega)$ inner product.\\
Correspondingly, we will use the Bochner-Sobolev space norms defined by
\begin{equation}\label{BSnorms}
\begin{aligned}
&\|v\|_{H^{\orti}(0,T;H^{s}(\Omega))}
=\left(\sum_{m\in\N} |m\omega|^{2\orti} \sum_{j\in\N} (\lambda^j)^s\sum_{k\in K^j} \left|\frac{2}{T}\int_0^T\langle \varphi^{j,k}, v(\cdot,t) \rangle\,e^{-\imath\omega t}\,dt\right|^2\right)^{1/2}
\\
&\|\vec{v}\|_{h^{\orti}(H^{s}(\Omega))}
=\left(\sum_{m\in\N} |m\omega|^{2\orti} \sum_{j\in\N} (\lambda^j)^s\sum_{k\in K^j} \left|\langle \varphi^{j,k}, \hat{v}_m \rangle\right|^2\right)^{1/2},
\end{aligned}
\end{equation}
where $H^{\orti}(0,T)$ denotes a Soloblev space of time periodic functions here. 
Note that with \eqref{u_multiharmonic}, we are also doing an eigenfunction expansion with respect to time. In fact, most of the analysis in this paper will work with space-time discretizations by tensor products of eigenfunctions. On the spatial side, this could be generalized to nonsymmetric operators $\mathcal{A}$, using the techniques from, e.g.,  \cite{JiangLiPauronYamamoto2023}. 

\subsubsection*{Main Results}
The first key result is local stability and uniqueness near a point $(\slo^0,\nlc^0,\uvec^0)$, where locality is quantified by the distance 
\begin{equation}\label{def_d}
\begin{aligned}
&d((\tilde{\slo},\tilde{\nlc},(\tilde{\uvec}_\no)_{\no=1}^2),(\slo^0,\nlc^0,(\uvec_\no^0)_{\no=1}^2))\\
&=\|(\tilde{\uvec}_\no)_{\no=1}^2-(\uvec^0_\no)_{\no=1}^2\|_{h^{\check{\orti}}(W^{\check{s},d/s}(\Omega)\cap L^{d/\check{\orti}}(\Omega))^2}
+\|\tilde{\slo}-\slo^0\|_{W^{\check{s},d/(s-\check{\orti})}(\Omega)\cap L^\infty(\Omega)} \\
&+\|(\mathcal{B}_m(\tilde{\uvec}_\no,\tilde{\uvec}_\no))_{\no=1}^2-(\mathcal{B}_m(\uvec^0_\no,\uvec^0_\no))_{\no=1}^2\|_{h^{\check{\orti}}(W^{\check{s},d/s}(\Omega)\cap L^{d/\check{\orti}}(\Omega))^2}\\
&+2\|(\tilde{\nlc}-\nlc^0)(\tilde{\uvec}_\no)_{\no=1}^2\|_{(h^{\check{\orti}}\cap\ell^1)(W^{\check{s},d/(s-\check{\orti})}(\Omega)\cap L^\infty(\Omega))^2}
\end{aligned}
\end{equation}
and an open neighborhood of $\xi^0=(\slo^0,0,(\uvec_\no^0)_{\no=1}^2)$ by  
\begin{equation}\label{Uc}
\begin{aligned}
U_c=\{&\tilde{\xi}=(\tilde{\slo},\tilde{\nlc},(\tilde{\uvec}_\no)_{\no=1}^2)\in\XXX\, :\, 
d((\tilde{\slo},\tilde{\nlc},(\tilde{\uvec}_\no)_{\no=1}^2),(\slo^0,0,(\uvec_\no^0)_{\no=1}^2))
%&\tilde{C}(s-\check{\orti},s,\Omega) \Bigl(
%\|(\tilde{\uvec}_\no)_{\no=1}^2-(\uvec^0_\no)_{\no=1}^2\|_{h^{\check{\orti}}(W^{\check{s},d/s}(\Omega)\cap L^{d/\check{\orti}}(\Omega))^2}
%+\|\tilde{\slo}-\slo^0\|_{W^{\check{s},d/(s-\check{\orti})}(\Omega)\cap L^\infty(\Omega)} \\
%&+\|(\mathcal{B}_m(\tilde{\uvec}_\no,\tilde{\uvec}_\no))_{\no=1}^2-(\mathcal{B}_m(\uvec^0_\no,\uvec^0_\no))_{\no=1}^2\|_{h^{\check{\orti}}(W^{\check{s},d/s}(\Omega)\cap L^{d/\check{\orti}}(\Omega))^2}\\
%&+2\|\tilde{\nlc}(\tilde{\uvec}_\no)_{\no=1}^2\|_{(h^{\check{\orti}}\cap\ell^1)(W^{\check{s},d/(s-\check{\orti})}(\Omega)\cap L^\infty(\Omega))^2}\Bigr)
< c\}.
\end{aligned}
\end{equation}
%with $d$ defined as in \eqref{def_d}.
%While $U_c$ is not a ball, $c>0$ plays the role of a size parameter which we will call the radius of $U$.

\begin{theorem}\label{thm:nlstab_mainresults}
There exist $(\slo^0,\nlc^0,(\uvec^0_\no)_{\no=1}^2)$ and $C(\tau)\,>0$ such that the stability estimate
 \begin{equation}\label{stabest_mainresults}
\begin{aligned}
&\|(\tilde{\slo},\tilde{\nlc},(\tilde{\uvec}_\no)_{\no=1}^2)-(\slo,\nlc,(\uvec_\no)_{\no=1}^2)\|_{H^s(\Omega)^2\times h^{\check{\orti}}(H^{s-\check{\orti}}(\Omega))^2}\\
&\leq C(\tau)\, 
\|\vec{F}((\tilde{\slo},\tilde{\nlc},(\tilde{\uvec}_\no)_{\no=1}^2))-\vec{F}((\slo,\nlc,(\uvec_\no)_{\no=1}^2))\|_{h^{\check{\orti}}(H^{s-\check{\orti}}(\Omega))^2\times\YYY^{obs}} 
\end{aligned}
\end{equation}
holds for all $(\slo,\nlc,\uvec)$, $(\tilde{\slo},\tilde{\nlc},\tilde{\uvec})\in U_{1/C(\tau)}$, 
%such that $d((\slo,\nlc,(\uvec_\no)_{\no=1}^2),(\slo^0,\nlc^0,(\uvec^0_\no)_{\no=1}^2))$, $d((\tilde{\slo},\tilde{\nlc},(\tilde{\uvec}_\no)_{\no=1}^2)),(\slo^0,\nlc^0,(\uvec^0_\no)_{\no=1}^2))$ $<1/C(\tau)$
with $\|\cdot\|_{\YYY^{obs}}$ according to \eqref{normYmodYobs}, $s>\frac12$, $\check{\orti}\in[0,\min\{s,1\}]$,
and 
{$\|\cdot\|_{h^{\check{\orti}}(H^{s-\check{\orti}}(\Omega))}$}
according to \eqref{BSnorms}.
\\
The function $\tau\mapsto C(\tau)$ is monotonically decreasing and tends to infinity as $\tau\to0$.
More precisely, it has the form $C(\tau)=2\max\{1,\overline{C}(\tau)\}$ with
\begin{equation}\label{olCtau_mainresults}
\overline{C}(\tau):=C_0\, 
\Bigl(\frac{(\alpha/\tau)}{1-e^{-2(\alpha/\tau)T_0}}\,\,e^{2(\alpha/\tau)\,(T-T_0)}
\Bigl(1+\left(\frac{\tau}{\beta}\right)^{\check{\orti}}\Bigr)
+1\Bigr)^{1/2}
\end{equation}
with $\alpha=\frac{\slo^0\beta-\tau}{2\beta}$, some $T_0\in(0,T)$, and a constant $C_0>0$ independent of $\alpha$, $\beta$ and $\tau$.
\end{theorem}

Although this is a Lipschitz estimate, it still reflects the possible ill-posendess of the inverse problem by the strength of the image space norm $\|\cdot\|_{\YYY^{obs}}$.

Theorem~\ref{thm:nlstab_mainresults} also reveals the fact that larger $\tau$ leads to a larger radius of uniqueness and stability as well as a better stability constant. This induces a possible homotopy type numerical solution strategy, starting at a large value of $\tau>0$ and successively reducing the relaxation time to the actual value $\tau_0\geq0$, in the sense of a quasi reversibility type regularization ~\cite{AmesClark_etal:1998,ClarkOppenheimer:1994,LattesLions:1969,Showalter:1974a,Showalter:1974b,Showalter:1976}. 
%as $\tau\searrow\tau_0$ for a realistic (low) value of $\tau_0$ or even 
This may even include the case $\tau_0=0$, which corresponds to the classical Westervelt model. %\cite{periodicWestervelt,periodicWest_2}. 
%Going all the way down to $\tau=0$ (that is considering reconstruction in the JMGT model as a quasi reversibility approach for reconstruction in the Westervelt) is prevented by the fact that the constants grow exponentially in $\tau$, while dependence of $\vec{F}$ on $\tau$ is only linear, cf. \eqref{IPfreq}, \eqref{Bm}.

With another $\tau$ dependent constant 
\begin{equation}\label{Ctiltau_mainresults}
\begin{aligned}
\widetilde{C}(\tau)=
C_1\,\left(\frac{(\alpha/\tau)}{1-e^{-2(\alpha/\tau)T_0}}\,\,e^{2(\alpha/\tau)\,(T-T_0)}
\left(\left(\frac{\beta}{\tau}\right)^{\check{\orti}}+1\right)\right)^{1/2}
\end{aligned}
\end{equation}
for some $C_1$ independent of $\alpha$, $\beta$, $\tau$, 
and the norm on the observation space defined by 
\[
\|\uld{\bar{\obs}}\|_{\widetilde{\YYY}^{obs}}:=\|\Tinv\uld{\bar{\obs}}\|_{W^{1,1}(0,T;H^{s+1}(\Omega))}, \\
\]
where $\Tinv$ is a particular right inverse of the trace operator $\text{tr}_\Sigma$, cf. \eqref{Tinv_ell}, \eqref{Tinv} 
we have 
\begin{equation}\label{Ytilobs_est_mainresults}
\|v\|_{\YYY^{obs}}\leq \widetilde{C}(\tau)\|v\|_{\widetilde{\YYY}^{obs}}\quad v\in \widetilde{\YYY}^{obs}\subseteq\YYY. 
\end{equation}
Combining this bound with \eqref{stabest_mainresults}, we obtain a following quasi reversibility result.
To this end, for noisy data satisfying
\begin{equation*}%\label{delta}
\|\obs_\no^\delta-\obs_\no\|_{\widetilde{\YYY}^{obs}}\leq\delta\, \quad \no\in\{1,2\}
\end{equation*}
we define a quasi reversibility regularized reconstruction by
\begin{equation*}%\label{Ftauydelta}
\vec{F}((\slo_\tau,\nlc_\tau,(\uvec_{\no,\tau})_{\no=1}^2)_\tau^\delta)=\vec{y}^\delta 
\end{equation*}
with 
%$\vec{y}=(\vec{r}_1,\vec{r}_2,\vec{\obs}_1,\vec{\obs}_2)$, 
$\vec{y}=(\vec{r}_1,\vec{r}_2,\vec{\obs}_1^\delta,\vec{\obs}_2^\delta)$ and $\vec{F}=\vec{F}_\tau$ given by \eqref{IPfreq}, \eqref{Bm}, \eqref{Fnlcuvecy}. 
The exact solution of the inverse problem with noiseless data and relaxation time $\tau_0\geq0$ is denoted by 
$(\slo_{\tau_0},\nlc_{\tau_0},(\uvec_{\no,\tau_0})_{\no=1}^2)$.

\begin{theorem}\label{thm:quasireversibility_mainresults}
Let $\partial_t u_{\no,\tau_0}\in H^{\check{\orti}}(0,T;H^{s-\check{\orti}}(\Omega))$,
$\no\in\{1,2\}$, 
with $s>\frac12$, $\check{\orti}\in[0,\min\{s,1\}]$, cf. \eqref{setting_s_orti},
 and either \\
(a) $\tau_0>0$ or \\
(b) $\tau_0=0$ and $T_0=T$ and $\check{\orti}<1$ \\
hold.
Then with a choice $\tau=\tau(\delta)$ such that 
\[
\tau(\delta)\to\tau_0, \quad \max\{\overline{C}(\tau),\,\widetilde{C}(\tau)\}\,\delta\to0, \quad 
\text{ as }\delta\to0
\]
convergence 
$
\|(\slo_{\tau(\delta)},\nlc_{\tau(\delta)},(\uvec_{\no,\tau(\delta)})_{\no=1}^2)^\delta-(\slo_{\tau_0},\nlc_{\tau_0},(\uvec_{\no,\tau_0})_{\no=1}^2)\|_{H^s(\Omega)^2\times h^{\check{\orti}}(H^{s-\check{\orti}}(\Omega))^2}\to0 \
$
as $\delta\to0$ holds.
\end{theorem}

\section{%Injectivity and 
Continuous invertibility of linearized all-at-once forward operator} \label{sec:uniqueness-lin}
We linearize $\vec{F}$ at a particularly chosen point $(\slo,\nlc,(\uvec^\no)_{\no=1}^2)$, namely 
\begin{equation}\label{u0separable}
\begin{aligned}
&\slo(x)\equiv\slo^0, \quad \nlc(x)\equiv0, \\ 
&\hat{u}^0_{\no,m}(x)=\phi(x) \hat{\psi}_{\no,m}  \ \text{ with }\phi\not=0\text{ a.e. in }\Omega, \\ 
&\hat{\psi}_{\no,m}\not=0, \ {\widehat{\psi_\no^2}}_m:=\mathcal{B}_m(\vec{\psi}_\no,\vec{\psi}_\no)\not=0, \ \no\in\{1,2\},\quad
\text{rank} (\mathfrak{M}_m)=2, \ m\in\N, 
\end{aligned}
\end{equation}
where the $2\times2$ complex matrices $\mathfrak{M}_m$ are defined by 
\[
\mathfrak{M}_m = \big((\mathfrak{M}_m)_{\no,q}\big)_{\no\in\{1,2\},q\in\{\slo,\nlc\}}
=\left(\begin{array}{cc}
\hat{\psi}_{1,m}&{\widehat{\psi_1^2}}_m\\
\hat{\psi}_{2,m}&{\widehat{\psi_2^2}}_m
\end{array}\right),
\]
and set
\begin{equation}\label{drdp}
(\uld{\rvec},\uld{\obsvec}):=\vec{F}'(\slo^0,0,\uvec^0)[(\uld{\slo},\uld{\nlc},\uld{\uvec})],
\end{equation}
where 
\begin{subequations}\label{IPfreq_diff}
\begin{alignat}{2}
{F_{\no,m}^{mod}}'(\slo^0,0,\uvec_\no^0)[\uld{\slo},\uld{\nlc},\uld{\uvec}_\no]=&
\mathcal{L}_m(\slo^0)\uld{u}_{\no,m}+\phi\,\uld{\slo}\,\hat{\psi}_{\no,m} 
+\phi^2\,\uld{\nlc}\, {\widehat{\psi_\no^2}}_m
%+\nlc\Bigl( \mathcal{B}_m(\uld{\uvec},\uvec)+\mathcal{B}_m(\uvec,\uld{\uvec})\Bigr)
\label{IPfreq_diff_mod}\\
{F_{\no,m}^{obs}}'(\slo^0,0,\uvec_\no^0)[\uld{\slo},\uld{\nlc},\uld{\uvec}_\no]=& \text{tr}_\Sigma\uld{u}_{\no,m}
\label{IPfreq_diff_obs}
\end{alignat}
\end{subequations}
%Note that $\vec{F}$ would not necessarily be differentiable at some $\eta\not=0$.
(Fr\'{e}chet differentiability can be proven by estimating the first order Taylor remainder in time domain and applying Parseval's identity; 
%also with nonzero $\nlc$
for our purposes, ${F_{\no,m}^{mod}}'$ does not need to be a derivative in the strong Fr\'{e}chet sense, though.)

Moreover, we use the following abbreviations
\begin{eqnarray}
\label{bmjkajk}
&&b_{\no,m}^{j,k}=\langle \uld{\hat{u}}_{\no,m},\varphi^{j,k}\rangle, \quad 
a_{\slo}^{j,k}=\langle \phi\,\uld{\slo},\varphi^{j,k}\rangle, \quad
a_{\nlc}^{j,k}=\langle \phi^2\,\uld{\nlc},\varphi^{j,k}\rangle, \
\\
\label{drjkm}
&&\uld{r}_{\no,m}^{j,k}=\langle \uld{\hat{r}}_{\no,m},\varphi^{j,k}\rangle={\tfrac{2}{T}}\int_0^T\langle\uld{r}_\no(t),\varphi^{j,k}\rangle \,e^{-\imath m \omega t}\, dt
\\
\label{thetaJMGT}
&&\vartheta(o)=\tau o^3+\slo^0 o^2, \qquad
\Theta(o)=\beta o +1, \qquad
o_m=\imath m\omega. 
%\quad {\widehat{\psi_\no^2}}_m=\mathcal{B}_m(\vec{\psi}_\no,\vec{\psi}_\no),
\end{eqnarray}
With this, we get from \eqref{IPfreq_diff_mod} with \eqref{Bm}, \eqref{u0separable}
\begin{equation}\label{projmod}
\begin{aligned}
&\forall \no\in\{1,2\}, \quad m\in\N, \quad j\in\N, \quad k\in K^j: \\ 
&\bigl(\vartheta(o_m)+\Theta(o_m)\lambda_j\bigr)  b_{\no,m}^{j,k} 
+o_m^2\, \Bigl(\sum_{q\in\{\slo,\nlc\}} (\mathfrak{M}_m)_{\no,q} \,a_q^{j,k}
- \uld{r}_{\no,m}^{j,k}\Bigr)=0,
\end{aligned}
\end{equation}
%Equation \eqref{projmod} 
which
can be resolved for $b_{\no,m}^{j,k}$
\[
\begin{aligned}
&\forall \no\in\{1,2\} \quad m\in\N\, \quad j\in\N, \quad k\in K^j: \quad \\ 
&b_{\no,m}^{j,k} = 
\frac{o_m^2}{\vartheta(o_m)+\Theta(o_m)\lambda_j}
\Bigl(\uld{r}_{\no,m}^{j,k} - \sum_{q\in\{c^2,\nlc\}} (\mathfrak{M}_m)_{\no,q} \,a_q^{j,k} \Bigr),
\end{aligned}
\]
This, in its turn, with the vector notation
\[
\bar{a}^{j,k}= (a_{\slo}^{j,k},\,a_{\nlc}^{j,k})^T, \quad
\bar{b}_m^{j,k}= (b_{1,m}^{j,k},b_{2,m}^{j,k})^T, \quad
\uld{\bar{r}}_m^{j,k}= (\uld{r}_{1,m}^{j,k},\uld{r}_{2,m}^{j,k})^T,
\]
can be rewritten as 
\begin{equation}\label{bmjk}
\forall  m\in\N\, \quad j\in\N, \quad k\in K^j: 
\quad  \bar{b}_m^{j,k} = 
\frac{o_m^2}{\vartheta(o_m)+\Theta(o_m)\lambda_j}
\Bigl(\uld{\bar{r}}_m^{j,k} - \mathfrak{M}_m \,\bar{a}^{j,k} \Bigr).
\end{equation}
With $\uld{\hat{u}}_{\no,m}=\sum_{j\in\N}\sum_{k\in K^j}b_{\no,m}^{j,k}\varphi^{j,k}$, the linearized observation equation \eqref{drdp}, \eqref{IPfreq_diff_obs} thus 
%reads as 
%\begin{equation}\label{linobs_coeff}
%\begin{aligned}
%&\forall m\in\N,\ x_0\in\Sigma\ : \quad \sum_{j\in\N}\frac{o_m^2}{\vartheta(o_m)+\Theta(o_m)\lambda_j}  \sum_{k\in K^j}  \text{tr}_\Sigma\varphi^{j,k}(x_0)
%\Bigl(\uld{\bar{r}}_m^{j,k} - \mathfrak{M}_m \bar{a}^{j,k}\Bigr)
% = \uld{\bar{\obs}}_m(x_0),
%\end{aligned}
%\end{equation}
%and 
after applying the inverse of $\mathfrak{M}_m$ 
yields 
\begin{equation}\label{redobs}
\begin{aligned}
&\forall m\in\N\ x_0\in\Sigma\ : \quad\\
&\sum_{j\in\N} \frac{o_m^2}{\vartheta(o_m)+\Theta(o_m)\lambda_j} \sum_{k\in K^j} \text{tr}_\Sigma\varphi^{j,k}(x_0)\Bigl(\mathfrak{M}_m^{-1}\uld{\bar{r}}_m^{j,k} - \bar{a}^{j,k}\Bigr)
 = \mathfrak{M}_m^{-1}\uld{\bar{\obs}}_m(x_0).
\end{aligned}
\end{equation}
%which is the analogon of
%\begin{equation}\label{redobs_nlc}
%\begin{aligned}
%&\forall m\in\N\ : \quad\\
%&-\sum_{j\in\N} \frac{o_m^2}{\vartheta(o_m)+\Theta(o_m)\lambda_j} \sum_{k\in K^j} \Bigl( a^{j,k} - %\tfrac{1}{\mathfrak{B}_m}\uld{r}_m^{j,k}\Bigr)
%\text{tr}_\Sigma\varphi^{j,k}(x) = \tfrac{1}{\mathfrak{B}_m} \uld{\obs}_m(x),
%\end{aligned}
%\end{equation}
%in the setting of \cite{nonlinear_imaging_JMGT_freq}.

The following result allows us to disentangle the individual terms in the sum over $j\in\mathbb{N}$, that is, the components in the individual eigenspaces $\mathbb{E}^j$.

\begin{lemma}\label{lem:m2j} \cite[Lemma 3.2]{nonlinear_imaging_JMGT_freq}\\
Given sequences $(\lambda_j)_{j\in\N}\subseteq\mathbb{R}$ with $\lambda_j\nearrow\infty$ as $j\to\infty$, $(\hat{\rho}^j_m)_{j,m\in\N}\subseteq\mathbb{C}$ and $(\hat{d}_m)_{m\in\N}\subseteq\mathbb{C}$, as well as a complex function $\Psi$, assume that $\widetilde{d}$, $\widetilde{\rho^j}$ interpolate $\hat{d}$, $\hat{\rho}$, that is, 
\begin{equation}\label{interpol}
\widetilde{d}(\imath m\omega)=\hat{d}_m, \quad \widetilde{\rho^j}(\imath m\omega)=\hat{\rho}^j_m, \quad m\in\N
\end{equation} 
such that the mappings
\begin{equation}\label{analytic}
z\mapsto z^2\Psi(\tfrac{1}{z}), \quad z\mapsto\widetilde{\rho^j}(\tfrac{1}{z}), \quad z\mapsto(1-\tfrac{1}{\Psi(\tfrac{1}{z})}\lambda_j) \widetilde{d}(\tfrac{1}{z}), \quad z\mapsto\prod_{k\in\N,k\not=j} (1-\frac{1}{\Psi(\tfrac{1}{z})}\lambda_k) \  j\in\N
\end{equation} 
define analytic functions on an open set $\mathbb{O}\subseteq\mathbb{C}$.
Moreover, we define the sequence of poles $(\pole_\ell)_{\ell\in\mathbb{N}}\subseteq\mathbb{C}$ by the identity  
$\Psi(\pole_\ell)=\lambda_\ell$,  
%but $\Psi(\pole_\ell)\not=\lambda_k$ for $k\not=\ell$ 
%THIS IMMEDIATELY FOLLOWS FROM $\lambda_\ell\not=\lambda_k$
and assume that 
\begin{equation}\label{D}
\mathbb{D}:=\{\tfrac{1}{\imath m\omega}\, : \, m\in\N\}\subseteq \mathbb{O}
\text{ and }\{z_\ell:=\tfrac{1}{\pole_\ell}\, : \, \ell\in\N\}\subseteq \overline{\mathbb{O}}.
\end{equation}
Then the implication
\[
\begin{aligned}
&\left(\forall m\in\N\, : \ 
\sum_{j\in\N}\frac{(\imath m\omega)^2}{\Psi(\imath m\omega)-\lambda_j}\,\bigl(c_j-\hat{\rho}^j_m\bigr) = \hat{d}_m\right)
\\ 
&\Rightarrow \
\left(\forall \ell\in\N\, : \ c_\ell = \frac{\Psi'({\pole_\ell})}{\pole_\ell^2} \, \text{res}(\widetilde{d};{\pole_\ell})+\widetilde{\rho^\ell}({\pole_\ell})\right),
\end{aligned}
\]
holds true,
where $\text{res}(f;p_\ell)=\lim_{o\to p_\ell} (o-p_\ell)\, f(o)$ is the residue of the function $f$ at $p_\ell$ if $p_\ell$ is a pole of $f$ and zero else.
\end{lemma}
We apply Lemma~\ref{lem:m2j} to \eqref{redobs}, which component wise (for each coefficient $q\in\{\slo,\,\nlc\}$) 
and for fixed $x_0\in \Sigma$ can be written as the premiss of this lemma with
\begin{equation}\label{defPsiJMGT} 
\begin{aligned}
%&\vartheta(o)=\tau o^3+\slo^0 o^2, \qquad \Theta(o)=\beta o +1,\\
&\Psi(o)
=-\frac{\vartheta(o)}{\Theta(o)}
=-\frac{\tau o^3+\slo^0 o^2}{\beta o +1}
=-\frac{\tau o^2+\frac{\slo^0 o}{\beta o+1}}{\beta+\frac{1}{\beta o+1}} 
%\todo{ \text{ in place of } 
%-\frac{o^2+\frac{\omzer o}{\tau o+1}}{c^2+\frac{\beta o}{\tau o+1}}}
=-o^2(\frac{\tau}{\beta}+\mathcal{O}(\tfrac{1}{o}))
\\
&c_j= \sum_{k\in K^j}  a_q^{j,k} \text{tr}_\Sigma\varphi^{j,k}(x_0),  
%\quad x_0\in \Sigma
\\
&\hat{d}_m(x_0)=\Theta(o_m)\,\bigl(\mathfrak{M}_m^{-1}\,\uld{\bar{\obs}}_m\bigr)_q(x_0)
, \quad{\hat{\rho}}^j_m(x_0)=\sum_{k\in K^j}\bigl(\mathfrak{M}_m^{-1}\,\uld{\bar{r}}_m^{j,k}\bigr)_q\text{tr}_\Sigma\varphi^{j,k}(x_0), 
%\quad x_0\in \Sigma
\\ 
&{\widetilde{d}}(o,x_0)= \Theta(o)\,\bigl(\widetilde{\mathfrak{M}}(o)^{-1}\,\widetilde{\uld{\bar{\obs}}}\bigr)_q(o,x_0)
, \ {\widetilde{\rho^j}}(o,x_0)=\sum_{k\in K^j}\bigl(\widetilde{\mathfrak{M}}(o)^{-1}\,\widetilde{\uld{\bar{r}}}^{j,k}(o)\bigr)_q\text{tr}_\Sigma\varphi^{j,k}(x_0),  
%\quad x_0\in \Sigma
\\
&\text{where for all }m\in\N\\
&
\widetilde{\uld{\bar{\obs}}}(\imath m\omega,x_0)=\uld{\bar{\obs}}_m(x_0),\quad
\widetilde{\uld{\bar{r}}}^{j,k}(\imath m\omega)=\uld{\bar{\obs}}_m^{j,k},\quad
\widetilde{\mathfrak{M}}(\imath m\omega)=\mathfrak{M}_m.
\end{aligned}
\end{equation}
The interpolants of $\uld{r}$ and $\mathfrak{M}$ can be explicitely written as
\begin{equation}\label{interpol_r}
\begin{aligned}
&\widetilde{\uld{r}}_\no^{j,k}(o)={\tfrac{2}{T}}\int_0^T\langle\uld{r}_\no(t),\varphi^{j,k}\rangle \,e^{-o t}\,dt,\\
&\widetilde{\mathfrak{M}}(o)_{\slo,\no}={\tfrac{2}{T}}\int_0^T \psi_\no(t) \,e^{-o t}\,dt,\quad
\widetilde{\mathfrak{M}}(o)_{\nlc,\no}={\tfrac{2}{T}}\int_0^T \psi_\no^2(t) \,e^{-o t}\,dt,
\end{aligned}
\end{equation}
which defines analytic functions, provided $\uld{r}$, $\psi_\no$ and $\psi_\no^2$ are integrable. 
As opposed to this, $\obsvec$ has poles. Therefore in view of the analyticity requirement \eqref{analytic} and the fact that $(1-\tfrac{1}{\Psi(o)}\lambda_j)=\tfrac{\Theta(o)}{\vartheta(o)}\,(\Psi(o)-\Psi(\pole_j))$, the definition of $\hat{d}_m$ in \eqref{defPsiJMGT}, as well as differentiability of $\Psi$, $\vartheta$, $\Theta$, and $\pole_j-o_m\not=0$, for all $j,m\in\mathbb{N}$, we actually analytically interpolate $o\mapsto (o-\pole_j)\uld{p}$, that is, 
\begin{equation}\label{interpol_p}
(o-\pole_j)\widetilde{\uld{\bar{\obs}}}^{j,k}(o)={\tfrac{2}{T}}\int_0^T(o-\pole_j)\,\langle\uld{\bar{\obs}}(t),\varphi^{j,k}\rangle \,e^{-o t}\,dt.
\end{equation}

Lemma~\ref{lem:m2j} 
%\cite[Lemma 3.2]{nonlinear_imaging_JMGT_freq} 
thus implies the vector valued identity
\begin{equation}\label{resobs}
\begin{aligned}
&\sum_{k\in K^\ell} \bar{a}^{\ell,k}  \text{tr}_\Sigma\varphi^{\ell,k}(x_0)\\
&= \widetilde{\mathfrak{M}}({\pole_\ell})^{-1}\left(
%-2\tau +\mathcal{O}(\lambda_\ell^{-1/2}) 
\frac{\Theta({\pole_\ell})\Psi'({\pole_\ell})}{\pole_\ell^2}
\, \text{res}(\widetilde{\uld{\bar{\obs}}}{(\cdot,x_0)};{\pole_\ell})
+\sum_{k\in K^\ell} \widetilde{\uld{\bar{r}}}^{\ell,k}({\pole_\ell})\text{tr}_\Sigma\varphi^{\ell,k}(x_0)\right)
\quad x_0\in\Sigma,
\end{aligned}
\end{equation}
and in particular also
\begin{equation}\label{resintrace}
x_0\mapsto\text{res}(\widetilde{\uld{\bar{\obs}}}{(\cdot,x_0)};{\pole_\ell})\in \text{tr}_\Sigma\mathbb{E}^\ell
\end{equation}
for all $\ell\in\mathbb{N}$.
%Here $\text{res}(f;\pole_\ell)=\lim_{o\to \pole_\ell} (o-\pole_\ell)\, f(o)$ is the residue of the function $f$ at $\pole_\ell$ if $\pole_\ell$ is a pole of $f$ and zero else. 
The poles $\pole_\ell$ are given by the equations $\Psi(\pole_\ell)=\lambda_\ell$, that is, by $\vartheta(\pole_\ell)+\Theta(\pole_\ell)\lambda_\ell=0$; see Lemma~\ref{lem:poles} below for more details on their location.

\medskip

By a quantified verson of unique continuation that can be obtained using the Bounded Inverse Theorem (cf. \cite[Lemma 3.3]{nonlinear_imaging_JMGT_freq}), for any $s>\frac12$ the trace operator restricted to each individual eigenspace
\[
\begin{array}{rlcl}\text{Tr}_\Sigma^{s,\ell}:
&(\mathbb{C}^{K^\ell},\|\cdot\|_{h^{s,\ell}})&\to& 
(\text{tr}_\Sigma(\mathbb{E}^\ell), {\|\cdot\|_{H^{s-1/2}(\Sigma)}})\\
&(f^k)_{k\in K^j}&\mapsto&\sum_{k\in K^\ell} f^k \text{tr}_\Sigma \varphi^{\ell,k}
\end{array} 
\]
is an isomorphism  between the given spaces.
Applying $(\text{Tr}_\Sigma^{s,\ell})^{-1}$ to \eqref{resobs} and using the fact that 
$(\text{Tr}_\Sigma^{s,\ell})^{-1}\, \text{tr}_\Sigma\, \text{Proj}_{\mathbb{E}^\ell} = \text{Proj}_{\mathbb{E}^\ell}$ as well as \eqref{interpol_r}, we obtain, for all $\ell\in\N$,
\begin{equation}\label{resobs_1}
\begin{aligned}
(\bar{a}^{\ell,k})_{k\in K^\ell}
=& 
\frac{\Theta({\pole_\ell})\Psi'({\pole_\ell})}{\pole_\ell^2}\,
(\text{Tr}_\Sigma^{{s,\ell}})^{-1}\,\left[\widetilde{\mathfrak{M}}({\pole_\ell})^{-1}
\, \text{res}(\widetilde{\uld{\bar{\obs}}};{\pole_\ell})\right]
\\&
+\widetilde{\mathfrak{M}}({\pole_\ell})^{-1}
%\text{Proj}_{{\mathbb{E}^\ell}} {\tfrac{2}{T}}\int_0^T \uld{\bar{r}}(t) \,e^{-{\pole_\ell} t}\,dt,
(\widetilde{\uld{\bar{r}}}^{\ell,k}(\pole_\ell))_{k\in K^\ell}
\end{aligned}
\end{equation}
and for all $\ell\in\N$, $m\in\N$, 
\begin{equation}\label{bmellk}
\begin{aligned}
(\bar{b}_m^{\ell,k})_{k\in K^\ell} =& 
-\frac{o_m^2}{\vartheta(o_m)+\Theta(o_m)\lambda_\ell}\Bigl\{
\frac{\Theta({\pole_\ell})\Psi'({\pole_\ell})}{\pole_\ell^2}\,
\mathfrak{M}_m \,
(\text{Tr}_\Sigma^{{s,\ell}})^{-1}\,\left[\widetilde{\mathfrak{M}}({\pole_\ell})^{-1}
\, \text{res}(\widetilde{\uld{\bar{\obs}}};{\pole_\ell})\right]
\\&\phantom{-\frac{o_m^2}{\vartheta(o_m)+\Theta(o_m)\lambda_\ell}\Bigl\{}
+\mathfrak{M}_m \, \widetilde{\mathfrak{M}}({\pole_\ell})^{-1}\,\widetilde{\uld{\bar{r}}}^{\ell,k}(\pole_\ell)
- \uld{\bar{r}}_m^{\ell,k}\Bigr\},
\end{aligned}
\end{equation}
where $(\widetilde{\uld{\bar{r}}}^{\ell,k}(\pole_\ell))_{k\in K^\ell}=
\text{Proj}_{{\mathbb{E}^\ell}} {\tfrac{2}{T}}\int_0^T \uld{\bar{r}}(t) \,e^{-{\pole_\ell} t}\,dt$, 
which via \eqref{bmjkajk} yields explicit formulas for the coefficients and states
\[
%\begin{aligned}
\phi\,\uld{\slo}=\sum_{j\in\N}\sum_{k\in K^j}a_{\slo}^{j,k}\varphi^{j,k}, \quad
\phi^2\uld{\nlc}=\sum_{j\in\N}\sum_{k\in K^j}a_{\nlc}^{j,k}\varphi^{j,k},\quad
\uld{\hat{u}}_{\no,m}=\sum_{j\in\N}\sum_{k\in K^j}b_{\no,m}^{j,k}\varphi^{j,k},
%\end{aligned}
\]
whose components can be used to define topologies in which the linearized inverse problem is stable. 
%as was done in \cite{nonlinear_imaging_JMGT_freq}.

More precisely, with Sobolev norms in pre-image space
\begin{equation}\label{SobolevX}
\begin{aligned}
\|(\uld{\slo},\uld{\nlc},\uld{\vec{\bar{u}}})\|_\XXX^2
&:=\|\phi\,\uld{\slo}\|_{H^s(\Omega)}^2+\|\phi^2\,\uld{\nlc}\|_{H^s(\Omega)}^2
+\|\uld{\bar{u}}\|_{H^{\check{\orti}}(0,T;H^{\check{s}}(\Omega))^2}^2\\
&:=\|\uld{\slo}\|_{H^s_{\phi}(\Omega)}^2+\|\uld{\nlc}\|_{H^s_{\phi^2}(\Omega)}^2
+\|\uld{\vec{\bar{u}}}\|_{h^{\check{\orti}}(H^{\check{s}}(\Omega))^2}^2,
\end{aligned}
\end{equation}
the reconstruction formulas \eqref{resobs_1}, \eqref{bmellk} lead us to using the image space norm
\begin{equation}\label{normYYY}
\|(\uld{\bar{p}},\uld{\bar{r}})\|_\YYY^2:= \|\uld{\bar{p}}\|_{\YYY^{obs}}^2+\|\uld{\bar{r}}\|_{\YYY^{mod}}^2
\end{equation}
defined by 
\begin{equation}\label{normYmodYobs}
\begin{aligned}
&\|\uld{\bar{p}}\|_{\YYY^{obs}}^2:=
\sum_{m\in\mathbb{N}} \sum_{\ell\in\mathbb{N}}
\frac{|o_m|^{4+2\check{\orti}}\lambda_\ell^{\check{s}}}{|\vartheta(o_m)+\Theta(o_m)\lambda_\ell|^2}\,\\
&\phantom{\|\uld{\bar{p}}\|_{\YYY^{obs}}^2:=\sum_{m\in\mathbb{N}}}\hspace*{1cm}
\sum_{k\in K^j}\left|\mathfrak{M}_m \frac{\Theta({\pole_\ell})\Psi'({\pole_\ell})}{\pole_\ell^2}\,
(\text{Tr}_\Sigma^{{s,\ell}})^{-1}\,\left[\widetilde{\mathfrak{M}}({\pole_\ell})^{-1}
\, \text{res}(\widetilde{\uld{\bar{\obs}}};{\pole_\ell})\right]\right|^2\\
&\phantom{\|\uld{\bar{p}}\|_{\YYY^{obs}}^2:=\sum_{m\in\mathbb{N}}} 
\sum_{\ell\in\mathbb{N}}\lambda_\ell^s
\sum_{k\in K^j}\left|\frac{\Theta({\pole_\ell})\Psi'({\pole_\ell})}{\pole_\ell^2}\,
(\text{Tr}_\Sigma^{{s,\ell}})^{-1}\,\left[\widetilde{\mathfrak{M}}({\pole_\ell})^{-1}
\, \text{res}(\widetilde{\uld{\bar{\obs}}};{\pole_\ell})\right]\right|^2\\
&\|\uld{\bar{r}}\|_{\YYY^{mod}}^2:=
\sum_{m\in\mathbb{N}} \sum_{\ell\in\mathbb{N}}
\frac{|o_m|^{4+2\check{\orti}}\lambda_\ell^{\check{s}}}{|\vartheta(o_m)+\Theta(o_m)\lambda_\ell|^2}
\sum_{k\in K^j}\left|\mathfrak{M}_m\,\widetilde{\mathfrak{M}}({\pole_\ell})^{-1}
\widetilde{\uld{\bar{r}}}^{\ell,k}(\pole_\ell)
- \uld{\bar{r}}_m^{\ell,k}\right|^2\\ 
&\phantom{\|\uld{\bar{r}}\|_{\YYY^{mod}}^2:=\sum_{m\in\mathbb{N}}}
+\sum_{\ell\in\mathbb{N}}\lambda_\ell^s
\sum_{k\in K^\ell}\left|\widetilde{\mathfrak{M}}({\pole_\ell})^{-1}
\widetilde{\uld{\bar{r}}}^{\ell,k}(\pole_\ell)\right|^2,
\end{aligned}
\end{equation}
which immediately implies injectivity of $\vec{F}'(\slo^0,0,\uvec^0)$ as well as the following linearized stability estimate.

\begin{proposition}\label{prop:linstab}
Under condition \eqref{u0separable}, with the norms defined in \eqref{SobolevX}, the estimate \eqref{normYmodYobs}
\begin{equation}\label{stabest_lincu}
\|(\uld{\slo},\uld{\nlc},\uld{\uvec})\|_{\XXX}
\leq \|\vec{F}'(\slo^0,0,\uvec^0)[(\uld{\slo},\uld{\nlc},\uld{\uvec})]\|_{\YYY}
\end{equation}
holds, that is, $\vec{F}'(\uld{\slo},0,\uvec^0)$ is injective with bound inverse $\|\vec{F}'(\uld{\slo},0,\uvec^0)^{-1}\|_{\YYY\to\XXX}\leq1$.
\end{proposition}

\begin{remark}[operator formulation in time domain]\label{rem:opform_lin}
With the operator representation
\[
F'(\slo^0,0,u^0)[(\uld{\slo},\uld{\nlc},\uld{u})]=
\left(\begin{array}{l}
\mathcal{L}(\slo^0)\uld{u}+\mathcal{M}_\psi[\phi\uld{\slo},\phi^2\uld{\nlc}]\\
\mathcal{T}\uld{u}
\end{array}\right)
\]
where 
\[
\begin{aligned}
&(\mathcal{M}_\psi[v_\slo,v_\nlc])(x,t)=
\left(\begin{array}{l}
v_\slo(x)\psi_1(t) + v_\nlc(x)\psi_1(t)^2\\
v_\slo(x)\psi_2(t) + v_\nlc(x)\psi_2(t)^2
\end{array}\right)
\\
&\mathcal{T}\uld{u}
=
\left(\begin{array}{l}
\text{tr}_\Sigma \uld{u}_1\\
\text{tr}_\Sigma \uld{u}_2
\end{array}\right)
\end{aligned}
\]
we can formally write the inversion steps 
\eqref{bmjk}, 
%\eqref{linobs_coeff}, 
\eqref{resobs} as 
\[
\begin{aligned}
&\uld{u}= \mathcal{L}(\slo^0)^{-1}\Bigl(\uld{r}-\mathcal{M}_\psi[\phi\uld{\slo},\phi^2\uld{\nlc}]\Bigr)\\
&\mathcal{T}\mathcal{L}(\slo^0)^{-1}\Bigl(\uld{r}-\mathcal{M}_\psi[\phi\uld{\slo},\phi^2\uld{\nlc}]\Bigr)=\uld{\obs}\\
&(\phi\uld{\slo},\phi^2\uld{\nlc})=(\mathcal{T}\mathcal{L}(\slo^0)^{-1}\mathcal{M}_\psi)^{-1}
\Bigl(\mathcal{T}\mathcal{L}(\slo^0)^{-1}\uld{r}-\uld{\obs}\Bigr)
\end{aligned}
\]
and it is in fact invertibility and boundedness of the inverse
\[
\begin{aligned}
&F'(\slo^0,0,u^0)^{-1} =\\
&\left(\!\!\!\begin{array}{cc}
(\mathcal{T}\mathcal{L}(\slo^0)^{-1}\mathcal{M}_\psi)^{-1}\mathcal{T}\mathcal{L}(\slo^0)^{-1} &
-(\mathcal{T}\mathcal{L}(\slo^0)^{-1}\mathcal{M}_\psi)^{-1}\\
\mathcal{L}(\slo^0)^{-1}(I-\mathcal{M}_\psi (\mathcal{T}\mathcal{L}(\slo^0)^{-1}\mathcal{M}_\psi)^{-1}\mathcal{T}\mathcal{L}(\slo^0)^{-1}) &
\mathcal{L}(\slo^0)^{-1}\mathcal{M}_\psi (\mathcal{T}\mathcal{L}(\slo^0)^{-1}\mathcal{M}_\psi)^{-1}
\end{array}\!\!\!\right)
\end{aligned}
\]
beween the spaces defined by the norms \eqref{SobolevX}, \eqref{normYmodYobs}, that we state in Proposition~\ref{prop:linstab}.
\end{remark}

\begin{remark}[generation of interior sources]\label{rem:u0separable}
Space-time separability of the reference state according to \eqref{u0separable} can in practice be achieved by a proper time staggering of the (boundary) excitation signal similarly to electronic focusing, see, e.g., \cite{Szabo2014}.

The following simple example illustrates the fact that indeed an excitation concentrated at a $d-1$ dimensional manifold $\Gamma\subseteq\overline{\Omega}$ immersed into the acoustic propagation domain or attached to its boundary (as realistically modeling, e.g., an array of piezoelectric transducers) is capable of producing such a reference state. To this end, note that in our all-at-once formulation, $u_0$ does not necessarily need to be a solution to the PDE model with coefficents $\slo(x)\equiv\slo^0$, $\nlc(x)\equiv0$; we rather construct it as an approximate solution to the PDE model with coefficents $\slo(x)\equiv\slo^0$, $\nlc(x)=\nlc_0(x)\not=0$, thereby exploiting generation of interior sources due to nonlinearity.
We choose $\phi$ as an eigenfunction corresponding to some nonzero eigenvalue $\lambda\not=0$ of $-\Delta$, equipped with the Robin boundary conditions $\partial_\nu u+\gamma u =0$ and set $\nlc_0(x):=\nlc^0/\phi(x)$ for some constant $\eta^0\in\mathbb{R}$. \footnote{Due to zeros of eigenfunctions, this may exhibit singularities and is therefore an element of $L^p(\Omega)$ only for some $p<\infty$; for this reason, some of the identities in this remark only hold in an almost everywhere sense and some of the computations done for this example are only formal. 
%An exception of this is the pure Neumann case $\gamma(x,t)\equiv0$, where $\phi\equiv\frac{1}{\text{meas}(\Omega)^{1/2}}$, but with $\lambda=0$, which leads to problems below
} 
Inserting this into the model part of \eqref{IPtime} with 
\[
\langle r_\no(t),w\rangle_{H^1(\Omega)^*,H^1(\Omega)}:=\int_\Gamma r^\Gamma_\no(t)\, w\, d\Gamma, \quad w\in H^1(\Omega)
\]
for some source $r^\Gamma_\no\in L^2(0,T;H^{-1/2}(\Gamma)$ yields
\begin{equation}\label{ibvp_phipsi}
\begin{aligned}
&\phi(x)\bigl(\tau\psi_\no'(t)+\slo^0\psi_\no(t) + \lambda \int_0^t\int_0^s\psi_\no(r)\, dr\, ds + \beta \lambda\int_0^t\psi_\no(s)\,ds + \nlc^0\, (\psi_\no^2)(t)\bigr) = 0
%\label{ODEpsi}
\\
&\partial_\nu \phi(x)\left(\int_0^t\int_0^s\psi_\no(r)\, dr\, ds+\beta\int_0^t\psi_\no(s)\,ds\right) =r^\Gamma_\no(x,t) \text{ on }\Gamma\times(0,T)
%\label{bndycondphipsi}
\\
&%(\tint^2 \psi_\no)(T)=0, \quad (\tint \psi_\no)(T)=0, \quad 
\psi_\no(T)=\psi_\no(0)
%\label{periodicpsi}
\end{aligned}
\end{equation}
The boundary conditions in \eqref{ibvp_phipsi} lead us to choosing the boundary source as the Neumann trace of $\phi$ on $\Gamma$, multiplied with the time dependent function  $\int_0^t\int_0^s(\psi_\no(r)+\beta\psi_\no'(r))\, dr\, ds$, 
once $\psi_\no$ has been determined from the first line of \eqref{ibvp_phipsi}. The latter can be done  by moving into frequency domain (thereby also taking into account temporal periodicity) as follows.
Making an approximate multiharmonic ansatz by not taking the real part in \eqref{u_multiharmonic} (which is in fact a well justified approximation, see, e.g. the numerical experiments in \cite{NikolicRauscher2025}, as well as the references therein) 
\[
\psi_\no(t)
= \sum_{k=1}^\infty \hat{\psi}_{\no,k}(x) e^{\imath k \omega t},
\]
we (approximately) transform the ODE resulting from the first line of \eqref{ibvp_phipsi} to 
\begin{equation}\label{ODEpsi_freq}
\begin{aligned}
%\tau\psi_\no'''(t)+\slo^0\psi_\no''(t) + \lambda \psi_\no(t) + \beta \lambda\psi_\no'(t) + \nlc^0\, (\psi_\no^2)''(t) = 0
&\Bigl(\lambda-\slo^0 m^2\omega^2\, +\imath m\omega(\beta\lambda-\tau m^2\omega^2)\Bigr)\hat{\psi}_{\no,m}
%\\&
=-\frac12 m^2\omega^2\eta^0 \sum_{j=1}^{m-1} \hat{\psi}_{\no,j} \hat{\psi}_{\no,m-j}\quad m\in\mathbb{N}.
\end{aligned}
\end{equation}
In the particular setting $\omega=\sqrt{\frac{\lambda}{\slo^0}}$, $\tau=\beta\lambda\omega^{-2}$, this indeed allows us to choose a nonzero value of $\hat{\psi}_{\no,1}\not=0$ for $m=1$, which through the right hand side terms in the triangular system \eqref{ODEpsi_freq} entails nonzero values of the higher harmonic components \\
$\hat{\psi}_{\no,m}=-\frac{m^2\omega^2\eta^0}{2(\lambda-\slo^0 m^2\omega^2\, +\imath m\omega(\beta\lambda-\tau m^2\omega^2))} \sum_{j=1}^{m-1} \hat{\psi}_{\no,j} \hat{\psi}_{\no,m-j}$, $m=2,\,3,\,4,\,\ldots$.\\
The functions $\vec{u}^0_{\no,m}(x)=\phi(x)\hat{\psi}_{\no,m}$ generated from a boundary source constructed this way are nonvanishing almost everywhere in $\Omega$.
%\\Since the above choice implies $\tau=\beta\slo^0$, it is in the limiting setting of the stability regime $\slo(x)\beta\geq\tau$.
\\
This construction also provides approximate space-time separable sources for the Westervelt case $\tau=0$ in case $\beta\slo^0$ is small enough.
%\\Note that we actually do not need $\hat{\psi}_{\no,m}\not=0$, $m\in\mathbb{N}$ in our uniqueness proof. Rather, the condition $\widetilde{\psi}(\pole_\ell)\not=0$, $\widetilde{\psi^2}(\pole_\ell)\not=0$ is essential, cf.~\eqref{u0separable}.
\end{remark}
%%%%%%%%%%%%%%%%%%%
\section{Local uniqueness and stability: 
%for the nonlinear inverse problem
proof of Therorem~\ref{thm:nlstab_mainresults}} \label{sec:uniqueness-nl}

To transfer linearized stability \eqref{stabest_lincu} to the fully nonlinear setting, we use an Inverse Function Theorem type argument together with smallness of the nonlinearity coefficient $\nlc$, which is a very natural assumption from the practical point of view.
To this end, we do not require continuous differentiability of $\vec{F}$ in a neighborhood of $\nlc=0$ though,  
%(which may in fact fail to hold at $\nlc\not=0$), 
but rather establish a Taylor remainder estimate 
\begin{equation}\label{Taylorremainder}
\text{err}_{Tay}:=\|\vec{F}(\tilde{\xi})-\vec{F}(\xi)-\vec{F}'(\xi^0)(\tilde{\xi}-\xi)\|_{\widetilde{\YYY}} \leq c \|\tilde{\xi}-\xi\|_\XXX , \quad \tilde{\xi},\,\xi\in U_c\subseteq\XXX
\end{equation}
for some $c\in(0,\|\vec{F}'(\xi^0)^{-1}\|_{\widetilde{\YYY}\to\XXX}^{-1})$, which by the the inverse triangle inequality yields
\begin{equation}\label{stabest0}
\|\vec{F}(\tilde{\xi})-\vec{F}(\xi)\|_{\widetilde{\YYY}} \geq 
\left(\frac{1}{\|\vec{F}'(\xi^0)^{-1}\|_{\widetilde{\YYY}\to\XXX}}-c\right) \|\tilde{\xi}-\xi\|_\XXX
 , \quad \tilde{\xi},\,\xi\in U_c
\end{equation}
where $\xi=(\slo,\nlc,\uvec)$, $\xi^0=(\slo^0,0,\uvec^0)$.
When setting $\YYY:={\widetilde{\YYY}}$ accordng to \eqref{normYYY}, \eqref{normYmodYobs}, we have $\|\vec{F}'(\xi^0)^{-1}\|_{\YYY\to\XXX}\leq1$ by definition. However, this norm is inconvenient to use for estimating the Taylor remainder, and we will therefore define ${\widetilde{\YYY}}$ by a Sobolev type norm on the model part. 
Indeed, due to linearity of the observations, only the model part yields a nontrivial contribution to the Taylor remainder \eqref{Taylorremainder}
\begin{equation}\label{errTay}
\begin{aligned}
\text{err}_{Tay}=&
\|\tilde{\slo} \,\tilde{\uvec}-\slo \,\uvec+(\tilde{\slo}-\slo)\,\uvec^0+\slo^0\,(\tilde{\uvec}-\uvec)
+\tilde{\nlc} \mathcal{B}_m(\tilde{\uvec},\tilde{\uvec})
\\&\quad
-\nlc \mathcal{B}_m(\uvec,\uvec)
-(\tilde{\nlc}-\nlc) \mathcal{B}_m(\uvec^0,\uvec^0)\|_{\YYY^{mod}}\\
=&\|\uld{\slo}\,(\tilde{\uvec}-\uvec^0)+(\tilde{\slo}-\slo^0)\,\uld{\uvec}
+\uld{\nlc} (\mathcal{B}_m(\tilde{\uvec},\tilde{\uvec}-\uvec^0)+\mathcal{B}_m(\tilde{\uvec}-\uvec^0,\uvec^0))
\\&\quad
+\nlc (\mathcal{B}_m(\tilde{\uvec},\uld{\uvec})+\mathcal{B}_m(\uld{\uvec},\uvec))\|_{\widetilde{\YYY}^{mod}}.
\end{aligned}
\end{equation}
for $\uld{\slo}=\tilde{\slo}-\slo$, $\uld{\nlc}=\tilde{\nlc}-\nlc$, $\uld{\uvec}=\tilde{\uvec}-\uvec$.
In order to achieve \eqref{Taylorremainder}, we aim to define $\widetilde{\YYY}^{mod}$ norm by a Sobolev norm dominating $\|\cdot\|_{\YYY^{mod}}$ 
%$H^{\ddot{\orti}}(0,T;H^{\ddot{s}}(\Omega))$ 
\begin{equation}\label{ddotnorm}
\|\uld{\bar{r}}\|_{\YYY^{mod}}\leq C(\tau) \|\uld{\bar{r}}\|_{H^{\ddot{\orti}}(0,T;H^{\ddot{s}}(\Omega))^2}
=:C(\tau) \|\uld{\bar{r}}\|_{\widetilde{\YYY}^{mod}},
\end{equation}
which with $\widetilde{\YYY}=\widetilde{\YYY}^{mod}\times\YYY^{obs}$ implies 
$\|\vec{F}'(\xi^0)^{-1}\|_{\widetilde{\YYY}\to\XXX}\leq \max\{1,C(\tau)\}$ and thus by \eqref{stabest0}
\begin{equation}\label{stabest1}
\|\tilde{\xi}-\xi\|_\XXX \leq 2\,\max\{1,C(\tau)\}\,\|\vec{F}(\tilde{\xi})-\vec{F}(\xi)\|_{\widetilde{\YYY}}, \quad 
\tilde{\xi},\xi\in U_{(2 \max\{1,C(\tau)\})^{-1}}.
\end{equation}
The order of time and space derivatives in the Taylor remainder estimate \eqref{errTay} together with \eqref{Taylorremainder} and the definition of $\|\cdot\|_\XXX$ in \eqref{SobolevX} entail the necessary conditions
\begin{equation}\label{cond_sorti_tay}
\ddot{\orti}\leq\check{\orti}, \quad \ddot{s}\leq\check{s}.
\end{equation}

\medskip

For establishing \eqref{ddotnorm}, we now prove several auxiliary results, each of whose purpose is highlighted in boldface.

\medskip

\noindent
First of all, it is crucial to investigate the \textbf{location of the poles}.
\begin{lemma}\label{lem:poles}
For $\tau>0$, the poles exhibit the asymptotic behaviour
\begin{equation}\label{poles_asymp}
\pole_\ell
= 
-\frac{\alpha}{\tau}
\pm
\sqrt{-\frac{\beta}{\tau}\lambda_\ell+2\frac{\alpha}{\tau\beta}+\frac{\alpha^2}{\tau^2}}\,
+\frac{\alpha}{\sqrt{\tau}}\,O(\lambda_\ell^{-1}) \text{ as }\ell\to\infty
\text{ with }
\alpha:=\frac{\slo^0\beta-\tau}{2\beta}\geq0.
\end{equation}
In particular we have 
\begin{equation}\label{poles}
0\leq-\Re(\pole_\ell)\leq\frac{\alpha}{\tau}\left(1+\frac{C}{\lambda_\ell}\right), \qquad
|\pole_\ell| {\leq\atop\geq} \sqrt{\frac{\beta}{\tau}\,\lambda_\ell}\,\left(1\pm \frac{C\alpha}{\lambda_\ell}\right)
\end{equation}
for some $C>0$ independent of $\ell$, $\tau$, $\alpha$, $\beta$.

This, for any $\ddot{\orti}\in
%\{0,1\}
{[0,1]}
$, $\ddot{s}\geq0$ implies existence of a constant $\Ccheck_{\ddot{\orti}}$ independent of $\ell$, $\tau$, $\alpha$, $\beta$, such that
\begin{equation}\label{norm_pole}
\begin{aligned}
&\lambda_\ell^{\ddot{s}+\ddot{\orti}}\,\sum_{k\in K^\ell} |\widetilde{\uld{r}}^{\ell,k}({\pole_\ell})|^2
=\|\sum_{k\in K^\ell} \widetilde{\uld{r}}^{\ell,k}({\pole_\ell})\,\varphi^{\ell,k}\|_{H^{\ddot{s}+\ddot{\orti}}(\Omega)}^2\\
&\leq\Ccheck_{\ddot{\orti}}\, 
\left(\frac{\tau}{\beta}\right)^{\ddot{\orti}}\,
e^{-2\Re(\pole_\ell) T}\, \|\text{Proj}_{{\mathbb{E}^\ell}}\uld{r}\|_{H^{\ddot{\orti}}(0,T;H^{\ddot{s}}(\Omega))}^2.
\end{aligned}
\end{equation}
\end{lemma}
\begin{proof} 
Rewriting the third order polynomial equation $\vartheta(\pole_\ell)+\Theta(\pole_\ell)\lambda_\ell=0$ as a perturbed quadratic one 
\[
\tau \beta^2 \pole_\ell^2 
+ (\sigma_0\beta-\tau)\,\beta\pole_\ell 
- (\sigma_0\beta-\tau)
+\beta^3\lambda_\ell'=0
\text{ where }\beta^3\lambda_\ell'
=\beta^3\lambda_\ell
%+\frac{\sigma_0\beta-\tau}{\beta\pole_\ell+1}  
+\frac{2\alpha}{\pole_\ell+1/\beta}  
\]
we obtain 
\[
\pole_\ell=
-\tfrac{\alpha}{\tau}
\pm
\sqrt{-\tfrac{\beta}{\tau}\lambda_\ell'+2\tfrac{\alpha}{\beta\tau}+\tfrac{\alpha^2}{\tau^2}}.
\]
Considering arbitrary accumulation points of $\rho_\ell:=\frac{\pole_\ell}{\sqrt{\lambda_\ell}}$ which satisfies 
\[
%\frac{\pole_\ell}{\sqrt{\lambda_\ell}}
\rho_\ell
=
-\frac{\alpha}{\tau\lambda_\ell}
\pm
\sqrt{-\frac{\beta}{\tau}-\frac{2\alpha}{\beta^2\tau}\,
\frac{1}{
%\frac{\pole_\ell}{\sqrt{\lambda_\ell}}
\rho_\ell
+\frac{1}{\beta\sqrt{\lambda_\ell}}}\, 
\frac{1}{\sqrt{\lambda_\ell}^3}
+(2\frac{\alpha}{\beta\tau}+\frac{\alpha^2}{\tau^2})\frac{1}{\lambda_\ell}}
\]
we conclude from $\lambda_\ell\to\infty$ that
$\frac{\pole_\ell}{\sqrt{\lambda_\ell}}\to\pm\imath\sqrt{\frac{\beta}{\tau}}$ as $\ell\to\infty$
%$\pole_\ell^{-1}=O(\lambda_\ell^{-1/2})$, 
%hence $|\lambda_\ell'-\lambda_\ell|=O(\lambda_\ell^{-1/2})$ and therefore
more precisely,
\[
\begin{aligned}
&\left|
-\tfrac{\alpha}{\tau}
\pm
\sqrt{-\tfrac{\beta}{\tau}\lambda_\ell+2\tfrac{\alpha}{\beta\tau}+\tfrac{\alpha^2}{\tau^2}}\,
-\pole_\ell\right|\\
&=\left(\sqrt{-\tfrac{\beta}{\tau}\lambda_\ell+2\tfrac{\alpha}{\beta\tau}+\tfrac{\alpha^2}{\tau^2}}
+\sqrt{-\tfrac{\beta}{\tau}\lambda_\ell'+2\tfrac{\alpha}{\beta\tau}+\tfrac{\alpha^2}{\tau^2}}\right)^{-1}\,
\tfrac{\beta}{\tau}|\lambda_\ell-\lambda_\ell'| \\
&\leq 
\sqrt{\tfrac{\beta}{\tau}}\, 
\left(\lambda_\ell-2\tfrac{\alpha}{\beta^2}-\tfrac{\alpha^2}{\beta\tau}\right)^{-1/2}\,
\tfrac{2\alpha}{\beta^3\pole_\ell+\beta^2},
\end{aligned}
\]
that is, \eqref{poles_asymp}, which directly implies \eqref{poles}.

To obtain \eqref{norm_pole}, we 
%combine \eqref{poles} with the estimates
use the Cauchy-Schwarz inequality in 
%cf. \cite[(27), (28)]{nonlinear_imaging_JMGT_freq}, 
\begin{equation}\label{est_interpol_r}
\begin{aligned}
&\|\sum_{k\in K^\ell} \widetilde{\uld{r}}^{\ell,k}({\pole_\ell})\,\varphi^{\ell,k}\|_{H^{\ddot{s}}(\Omega)}
= \frac{2}{T}\,\|\int_0^T\text{Proj}_{\mathbb{E}^\ell}\uld{r}(t) \,e^{-{\pole_\ell} t}\,dt\|_{H^{\ddot{s}}(\Omega)}\\
&\leq \frac{2}{T}\,
\sqrt{\frac{e^{-2\Re({\pole_\ell}) T}-1}{-2\Re({\pole_\ell})}}
\|\text{Proj}_{\mathbb{E}^\ell}\uld{r}\|_{L^2(0,T;H^{\ddot{s}}(\Omega))}
\end{aligned}
\end{equation}
and integration by parts in 
\begin{equation}\label{est_interpol_r1}
\begin{aligned}
&\|\sum_{k\in K^\ell} \widetilde{\uld{r}}^{\ell,k}({\pole_\ell})\,\varphi^{\ell,k}\|_{H^{1+\ddot{s}}(\Omega)}\\
&=\frac{2\sqrt{\lambda_\ell}}{T}\,\|-\int_0^T\text{Proj}_{\mathbb{E}^\ell}\uld{r}_t(t) \,\tfrac{e^{-{\pole_\ell} t}}{-{\pole_\ell}}\,dt
+\left[\text{Proj}_{\mathbb{E}^\ell}\uld{r}(t) \,\tfrac{e^{-{\pole_\ell} t}}{-{\pole_\ell}}\right]_0^T\|_{H^{\ddot{s}}(\Omega)}\\
&\leq \frac{2\sqrt{\lambda}_\ell}{|{\pole_\ell}|\,T}\,
\bigl(\sqrt{\frac{e^{-2\Re({\pole_\ell}) T}-1}{-2\Re({\pole_\ell})}}
+ C_{H^1,L^\infty}^{(0,T)}\, e^{-\Re({\pole_\ell}) T}\bigr)
\|\text{Proj}_{\mathbb{E}^\ell}\uld{r}\|_{H^1(0,T;H^{\ddot{s}}(\Omega))},
\end{aligned}
\end{equation}
which we combine with \eqref{poles} to obtain \eqref{norm_pole} for $\ddot{\orti}\in\{0,1\}$ with
\begin{equation}\label{Ctil}
%\begin{aligned}
\Ccheck_{0}=\frac{2}{T}\, \sup_{x>0}\frac{1-e^{-x}}{x},\qquad
\Ccheck_{1}= \Ccheck_{0} + \frac{2}{T}\, C_{H^1,L^\infty}^{(0,T)}. 
%\end{aligned}
\end{equation} 
{The assertion for intermediate values $\ddot{\orti}\in[0,1]$ follows by interpolation.}
\end{proof}
Boundedness \eqref{poles} of the negative real parts of the poles is essential for stability of the inverse problem.
We refer to \cite{nonlinear_imaging_JMGT_freq} for a comparison of JMGT to Westervelt and some other models with respect to the location of poles and the consequences for ill-posedness of the inverse problem of imaging $\nlc$.
We track dependence on $\alpha$, $\beta$ and even more importantly on $\tau$ in these and the following estimates, in view of the possible role of $\tau>0$ as a regularization parameter, 
as will be discussed in the next section.

\medskip

The \textbf{choice of the reference sources $\psi_\no$, $\no\in\{1,2\}$} is another important ingredient of our stability proof
%following possible 
that will allow us to further bound the $\YYY^{mod}$ norm. 
Here we make use of the fact that due to nonlinearity, an amplitude modulation provides additional information and, for some 
%$\psi\in L^p(0,T)$ (with requirements on $p\in[1,\infty]$ to be encountered a few lines below) 
$\psi\in L^4(0,T)$ (as will become plausible a few lines below) 
and $A\in\mathbb{R}\setminus\{0,1\}$, set 
\begin{equation}\label{amplitudemodulation}
\psi_1=\psi, \quad \psi_2=A\psi,
\end{equation}
so that 
\[
\widetilde{M}(o)=\left(\begin{array}{cc} 
\widetilde{\psi}(o)&\widetilde{\psi^2}(o)\\
A\widetilde{\psi}(o)&A^2\widetilde{\psi^2}(o)
\end{array}\right).
\]
Since we will need summability of $\mathfrak{M}_m$ and due to Parseval's identity
\[
\begin{aligned}
\sum_{m\in\mathbb{N}}\|\mathfrak{M}_m\|_F^2
&=(1+A^2)\sum_{m\in\mathbb{N}}|\widetilde{\psi}(\imath m\omega)|^2
+(1+A^4)\sum_{m\in\mathbb{N}}|\widetilde{\psi^2}(\imath m\omega)|^2\\
&=(1+A^2)\|\psi\|_{L^2(0,T)}^2
+(1+A^4)\|\psi^2\|_{L^2(0,T)}^2
\end{aligned}
\]
holds, by equivalence of the spectral and the Frobenius norm of $2\times 2$ complex matrices, we find $\psi\in L^4(0,T)$ to be a sufficient condition for this purpose.

On the other hand, we also need to bound $\widetilde{M}(\pole_\ell)^{-1}$, which, according to Cramer's rule, is given by
\[
\widetilde{\mathfrak{M}}(\pole_\ell)^{-1}
=\frac{1}{A(A-1)\,\widetilde{\psi}(\pole_\ell)\,\widetilde{\psi^2}(\pole_\ell)}
\left(\begin{array}{cc} 
A^2 \widetilde{\psi^2}(\pole_\ell)&-\widetilde{\psi^2}(\pole_\ell)\\
A\widetilde{\psi}(\pole_\ell)&-\widetilde{\psi}(\pole_\ell)
\end{array}\right),
\]
hence 
\[
\|\widetilde{\mathfrak{M}}(\pole_\ell)^{-1}\|_F^2
=\frac{1}{A^2(A-1)^2}\Bigl(
(A^4+1)\frac{1}{|\widetilde{\psi}(\pole_\ell)|^2}
+(A^2+1)\frac{1}{|\widetilde{\psi^2}(\pole_\ell)|^2}
\Bigr)
\]
In here, for $\psi\in L^4(0,T)$, we have, using the Cauchy-Schwarz inequality and $\Re(\pole_\ell)\geq0$,
\[
\begin{aligned}
\frac{1}{|\widetilde{\psi^2}(\pole_\ell)|^2}
&=\left|\frac{2}{T}\int_0^T\psi^2(t)e^{-\pole_\ell\,t}\,dt\right|^{-2}
\geq \frac{T^2}{4\|\psi\|_{L^4(0,T)}^4}\left(\int_0^Te^{-2\Re(\pole_\ell)\,t}\,dt\right)^{-1}
\\
&=:\frac{T^2}{4\|\psi\|_{L^4(0,T)}^4} \Rho_T(2\Re(\pole_\ell))
=\frac{T^2}{2\|\psi\|_{L^4(0,T)}^4}
\frac{|\Re(\pole_\ell)|\,e^{2\Re(\pole_\ell)\,T}}{|1-e^{2\Re(\pole_\ell)\,T}|}.
\end{aligned}
\]
and likewise $\frac{1}{|\widetilde{\psi}(\pole_\ell)|^2}\geq \frac{T^2}{2\|\psi\|_{L^2(0,T)}^2}\frac{|\Re(\pole_\ell)|\,e^{2\Re(\pole_\ell)\,T}}{|1-e^{2\Re(\pole_\ell)\,T}|}$.
Thus, again using equivalence of the spectral and the Frobenius norm on $\mathbb{C}^{2\times2}$, with constant $C_{F,2}=1$ (due to submultiplicatviy of the Frobenius norm)
%and continuity of the embedding $L^4(0,T)\to L^2(0,T)$  
we obtain the following result.
\begin{lemma}\label{lem:A}  
Let $\psi\in L^4(0,T)$. Then with $\psi_1$, $\psi_2$ chosen according to \eqref{amplitudemodulation}, we have  
\[
\sum_{m\in\mathbb{N}}\|\mathfrak{M}_m\|^2\leq C_\psi:=C_{F,2}^2\,(A^4+A^2+2)\, \max\{\|\psi\|_{L^4(0,T)}^4,\|\psi\|_{L^2(0,T)}^2\}
\]
%with the equivalence constant $C_{F,2}$ between the spectral and the Frobenius norm on $\mathbb{C}^{2\times2}$ 
and 
\begin{equation}\label{muell}
\begin{aligned}
&\|\widetilde{\mathfrak{M}}(\pole_\ell)^{-1}\|^2\leq \|\widetilde{\mathfrak{M}}(\pole_\ell)^{-1}\|_F^2
=\mu_\ell^2 \, \Rho_T(2\Re(\pole_\ell)),\\
&\Rho_T(-x)=\left(\int_0^Te^{xt}\, dt\right)^{-1}= 
\frac{x}{1-e^{-x\,T}}\, e^{-xT}, \ x\geq0
%\frac{-2\Re(\pole_\ell)}{1-e^{2\Re(\pole_\ell)\,T}}\,e^{2\Re(\pole_\ell)\,T}
\\
&\text{ where }\mu_\ell^2\geq \frac{A^4+A^2+2}{4A^2(A-1)^2}\,T^2 \min\{\|\psi\|_{L^4(0,T)}^4,\|\psi\|_{L^2(0,T)}^2\}^{-1}. 
%\\&\text{ and }\frac{|\Re(\pole_\ell)|\,e^{2\Re(\pole_\ell)\,T}}{(1-e^{2\Re(\pole_\ell)\,T})}\leq \frac12,
\end{aligned}\end{equation}
\end{lemma}
%where the latter can be obtained by considering the limits at $0$, $+\infty$ and the critical point of the elementary function $x\mapsto\frac{x}{e^x-1}$ on the positive real line.
\begin{remark}\label{rem:ideal_psi} 
Note that so far we have only used the property $\Re(\pole_\ell)\leq0$ of the poles.
Choosing $\psi$ as an approximation of a delta pulse concentrated at $T_0\in(0,T]$, 
cf. \cite[Section 3.3]{nonlinear_imaging_JMGT_freq} 
%\textcolor{green}{(where  periodicity of $\psi$ enforces $T_0<T$),}
one can in fact achieve \eqref{muell} with $\mu_\ell$ uniformly bounded
\begin{equation}\label{muellbdd}
\|\widetilde{\mathfrak{M}}(\pole_\ell)^{-1}\|^2
\leq C_\mu^2 \, \Rho_{T_0}(2\Re(\pole_\ell)),
%\frac{-2\Re(\pole_\ell)}{1-e^{2\Re(\pole_\ell)\,T_0}} \,e^{2\Re(\pole_\ell)\,T_0}
\, \quad \ell\in\mathbb{N}
\end{equation}
%\textcolor{red}{There is an error in the derivation of \cite[Section 3.3]{nonlinear_imaging_JMGT_freq}: $L^p(0,T)$ is not dense in $\mathcal{M}(0,T)$, for otherwise $\mathcal{M}(0,T)$ would be separable, and so would be its pre-dual $C(0,T)$;}
%\textcolor{cyan}{Put Gaussian approximations of delta pulses here; approximate them in $\text{span}\{t\mapsto\sin(m\omega t),\ t\mapsto\cos(m\omega t)\,:\, m\in\mathbb{N}_0\}$; since the latter approximation is done in $L^4(\Omega)$, we don't need to impose the zero mean property nor periodicity on the Gaussians}
\end{remark}
 
Combination of Lemma~\ref{lem:A} with \eqref{norm_pole} and \eqref{muellbdd} yields
\begin{equation}\label{lemAnormpole}
\begin{aligned}
&\lambda_\ell^{\ddot{s}+\ddot{\orti}}\,\sum_{k\in K^\ell} \left|\widetilde{\mathfrak{M}}(\pole_\ell)^{-1}\widetilde{\uld{r}}^{\ell,k}({\pole_\ell})\right|^2\\
&\leq \Ccheck_{\ddot{\orti}}\, C_\mu^2 \, \left(\frac{\tau}{\beta}\right)^{\ddot{\orti}}
\frac{-2\Re(\pole_\ell)}{1-e^{2\Re(\pole_\ell)\,T_0}} \,e^{-2\Re(\pole_\ell)\,(T-T_0)}
\,\|\text{Proj}_{{\mathbb{E}^\ell}}\uld{r}\|_{H^{\ddot{\orti}}(0,T;H^{\ddot{s}}(\Omega))}^2\\
&\leq \Ccheck_{\ddot{\orti}}\, C_\mu^2 \, 
%\alpha\,\beta^{-\ddot{\orti}}(1-e^{-2(\alpha/\tau)\,T_0})^{-1}
\frac{\alpha}{\beta^{\ddot{\orti}}(1-e^{-2(\alpha/\tau)\,T_0})}
\,\tau^{\ddot{\orti}-1}
\,e^{2(\alpha/\tau)\,(T-T_0)}
\,\|\text{Proj}_{{\mathbb{E}^\ell}}\uld{r}\|_{H^{\ddot{\orti}}(0,T;H^{\ddot{s}}(\Omega))}^2.
\end{aligned}
\end{equation}
\medskip

Finally, in order to \textbf{control amplification factors}, we need to bound the real function defined by 
\[
\mathcal{J}^\chi_m(\lambda):= |\vartheta(o_m)+\Theta(o_m)\lambda_\ell|^2\lambda^{\chi}
=(a^2\lambda^2-2b\lambda+d^2)\lambda^\chi
\] 
with 
\[
\begin{aligned}
&a^2=|\Theta(o_m)|^2=\beta^2\,|o_m|^2+1, \\
&b=-\Re(\vartheta(o_m)\overline{\Theta(o_m)})=\tau\beta|o_m|^4+\slo^0|o_m|^2, \\
&d^2=|\vartheta(o_m)|^2=\tau^2\,|o_m|^6+(\slo^0)^2\,|o_m|^4,
\end{aligned}
\]
from below.

%In case $\tau=0$, we have $\mathcal{J}^\chi_m(\lambda)= \beta^2\,|o_m|^2\,\lambda^2+(\slo^0|o_m|^2-\lambda)^2$

\begin{lemma}\label{lem:Chat}
For $\vartheta$, $\Theta$, $o_m$ as defined in \eqref{thetaJMGT} and any $\chi\geq0$, the function
$\mathcal{J}^\chi_m$ satisfies 
\begin{equation}\label{bound_omJm}
\forall \ell\,, m\, \in\mathbb{N}\, : \ \frac{|o_m|^{4+2\chi}}{\mathcal{J}^\chi_m(\lambda_\ell)}
\leq \frac{2+\chi}{2(\slo^0)^2}\left(\frac{\beta}{\tau}\right)^\chi\left(1+\frac{1}{\beta^2\omega^2}\right)
\left(1-\frac{\tau}{\beta\slo^0}\right)^{-2}=: 
\Chat_\chi\,\left(\frac{\beta}{\tau}\right)^\chi.
\end{equation}
\begin{comment}
Moreover, the minimum of $\mathcal{J}^\chi_m$ is attained at 
\[
\begin{aligned}
&\lambda_*^m:=\frac{(\chi+1)b+\sqrt{D}}{(\chi+2)a^2}=\frac{\tau}{\beta}|o_m|^2(1+O(|o_m|^{-1})), \text{ where }\\ 
&D=a^2d^2-(\chi+1)^2(a^2d^2-b^2)=\tau^2\beta^2\,|o_m|^8(1+O|o_m|^{-1})\\
\end{aligned}
\]
and for any  $\rho\in(1,2)$, 
the perturbation estimate
\begin{equation}\label{perturbest}
\begin{aligned}
&\mu|o_m|^{2-\chi+\rho/2} \leq |\uld{\lambda}| \leq \,\frac{\tau}{2\beta}|o_m|^2\\
&\Rightarrow \quad 
\mathcal{J}^\chi_m(\lambda_*^m+\uld{\lambda})-\mathcal{J}^\chi_m(\lambda_*^m)\geq \frac12\,\tau^\chi \beta^{2-\chi} \,\mu^2\,|o_m|^{6+\rho}(1-|o_m|^{-1})
\end{aligned}
\end{equation}
holds .
\end{comment}
\end{lemma}
\begin{proof}
The elementary computations for determining the minimizer of $\mathcal{J}^\chi_m$ can be found in \cite[Section 3.3]{nonlinear_imaging_JMGT_freq}.
\begin{comment}
Some further computations, using ${\mathcal{J}^\chi_m}'(\lambda_*^m)=0$ and the notation $\approx$ $\gtrsim$, $\lesssim$ for approximations up to an $O(\lambda_\ell^{-1})$ error, yields the identities
\[
\begin{aligned}
&{\mathcal{J}^\chi_m}''(\lambda_*^m)
%=2\sqrt{D}\geq 2\tau\beta|o_m|^4(1-O(|o_m|^{-1}), 
={\mathcal{J}^\chi_m}''(\lambda_*^m)-\frac{\chi+1}{\lambda_*^m} {\mathcal{J}^\chi_m}'(\lambda_*^m)
=(2(\chi+1)b\lambda_*^m-2\chi d^2) (\lambda_*^m)^{\chi-2}\\
&\phantom{{\mathcal{J}^\chi_m}''(\lambda_*^m)}\gtrsim 2\tau^\chi \beta^{2-\chi}\, |o_m|^{2+2\chi}
\\
%&{\mathcal{J}^\chi_m}'''\equiv 6 d^2= 6(\beta^2\,|o_m|^2+1)
&{\mathcal{J}^\chi_m}'''(\lambda_*^m)
={\mathcal{J}^\chi_m}'''(\lambda_*^m)-\frac{(\chi+1)\chi}{(\lambda_*^m)^2} {\mathcal{J}^\chi_m}'(\lambda_*^m)
=2(2(\chi+1)b\lambda_*^m-(2\chi-1) d^2) (\lambda_*^m)^{\chi-3}
\\
&\phantom{{\mathcal{J}^\chi_m}'''(\lambda_*^m)}\lesssim 6 \tau^{\chi-1} \beta^{3-\chi}\, |o_m|^{2\chi}
\end{aligned}
\]
hence 
\[
\frac{{\mathcal{J}^\chi_m}'''(\lambda_*^m)}{{\mathcal{J}^\chi_m}''(\lambda_*^m)}
\lesssim \frac13\, \frac{\beta}{\tau}\, |o_m|^2,
\]
which via a Taylor expansion with ${\mathcal{J}^\chi_m}'(\lambda_*^m)=0$ 
\[
\mathcal{J}^\chi_m(\lambda_*^m+\uld{\lambda})-\mathcal{J}^\chi_m(\lambda_*^m)
=\left(\frac12{\mathcal{J}^\chi_m}''(\lambda_*^m)+\frac16{\mathcal{J}^\chi_m}'''(\lambda_*^m+\xi\uld{\lambda})\uld{\lambda}\right)\,\uld{\lambda}^2
\]
for some $\xi\in[0,1]$
imply \eqref{perturbest}.
\end{comment}
\end{proof}
The assertion of Lemma~\ref{lem:Chat} remains valid for $\tau=0$ in case $\chi=0$ (but not in case $\chi>0$), cf. \cite{nonlinear_imaging_JMGT_freq}.

\medskip

Combining these lemmas 
and assuming \eqref{muellbdd}
cf. Remark~\ref{rem:ideal_psi}
we can estimate the $\|\cdot\|_{\YYY^{mod}}$ norm by a Bochner Sobolev norm \eqref{ddotnorm}
as follows 
\begin{equation}\label{BSnormest}
\begin{aligned}
\|\uld{\bar{r}}\|_{\YYY^{mod}}^2
\leq&\sum_{m\in\mathbb{N}} \sum_{\ell\in\mathbb{N}}
\frac{|o_m|^{4+2\check{\orti}}\lambda_\ell^{\check{s}}}{|\vartheta(o_m)+\Theta(o_m)\lambda_\ell|^2}
\sum_{k\in K^\ell}\left|\mathfrak{M}_m\,\widetilde{\mathfrak{M}}({\pole_\ell})^{-1}
\widetilde{\uld{\bar{r}}}^{\ell,k}(\pole_\ell)\right|^2\\
&+\sum_{m\in\mathbb{N}} \sum_{\ell\in\mathbb{N}}
\frac{|o_m|^{4+2\check{\orti}}\lambda_\ell^{\check{s}}}{|\vartheta(o_m)+\Theta(o_m)\lambda_\ell|^2}
 \left|\uld{\bar{r}}_m^{\ell,k}\right|^2\\ 
&+\sum_{\ell\in\mathbb{N}}\lambda_\ell^s
\sum_{k\in K^\ell}\left|\widetilde{\mathfrak{M}}({\pole_\ell})^{-1}
\widetilde{\uld{\bar{r}}}^{\ell,k}(\pole_\ell)\right|^2\\
\leq& 
\Ccheck_{\ddot{\orti}}\, C_\mu^2 \, 
%\alpha\,\beta^{-\ddot{\orti}}(1-e^{-2(\alpha/\tau)\,T_0})^{-1}
\frac{\alpha}{\beta^{\ddot{\orti}}(1-e^{-2(\alpha/\tau)\,T_0})}
\,\tau^{\ddot{\orti}-1}\,e^{2(\alpha/\tau)\,(T-T_0)}\,
\\
&\cdot\sum_{\ell\in\mathbb{N}}\left(
\sum_{m\in\mathbb{N}} 
\frac{|o_m|^{4+2\check{\orti}}\lambda_\ell^{-(s-\check{s})}}{|\vartheta(o_m)+\Theta(o_m)\lambda_\ell|^2} |\mathfrak{M}_m|^2 +1\right)
\|\text{Proj}_{{\mathbb{E}^\ell}}\uld{r}\|_{H^{\ddot{\orti}}(0,T;H^{s-\ddot{\orti}}(\Omega))}^2\\
&+\Chat_0\, \|\uld{\bar{r}}\|_{H^{\check{\orti}}(0,T;H^{\check{s}}(\Omega))}^2,
\end{aligned}
\end{equation}
where $\ddot{\orti}\in%\{0,1\}
{[0,1]}
$ and we have used Lemmas~\ref{lem:A}, \ref{lem:Chat} and \eqref{lemAnormpole}.

Assuming $\check{\orti}\leq s-\check{s}$ we can use Lemmas~\ref{lem:A}, \ref{lem:Chat} again to obtain
\begin{equation}\label{pre-ddotnorm_Sob}
\begin{aligned}
\|\uld{\bar{r}}\|_{\YYY^{mod}}^2
\leq& 
\Ccheck_{\ddot{\orti}}\, C_\mu^2 \, 
%\alpha\,\beta^{-\ddot{\orti}}(1-e^{-2(\alpha/\tau)\,T_0})^{-1}
\frac{\alpha}{\beta^{\ddot{\orti}}(1-e^{-2(\alpha/\tau)\,T_0})}
\,\tau^{\ddot{\orti}-1}\,e^{2(\alpha/\tau)\,(T-T_0)}\,
\\
&\cdot\Bigl(C_\psi\,\Chat_{s-\check{s}}\,\left(\frac{\beta}{\tau}\right)^{s-\check{s}}
+1\Bigr) \,
\|\uld{\bar{r}}\|_{H^{\ddot{\orti}}(0,T;H^{s-\ddot{\orti}}(\Omega))}^2\\
&+\Chat_0 \|\uld{\bar{r}}\|_{H^{\check{\orti}}(0,T;H^{\check{s}}(\Omega))}^2.
\end{aligned}
\end{equation}
In order to obtain \eqref{ddotnorm} from this, the norm in the last term leads us to imposing
$\ddot{\orti}\geq\check{\orti}$, $\ddot{s}\geq\check{s}$, which together with \eqref{cond_sorti_tay} leads to
$\ddot{\orti}=\check{\orti}$, $\ddot{s}=\check{s}$, and the norm in the first term to $s-\ddot{\orti}\leq\ddot{s}$.
This, together with the above assumptions $s>\frac12$, $\check{\orti}\leq s-\check{s}$ made above, admits the choice
\begin{equation}\label{setting_s_orti}
%s=1, \quad \ddot{\orti}=\check{\orti}=1, \quad \ddot{s}=\check{s}=0,  \qquad
s>\frac12, \quad 0\leq\ddot{\orti}=\check{\orti}\leq1, \quad 0\leq\ddot{s}=\check{s}=s-\check{\orti}\leq s,
\end{equation}
which yields \eqref{ddotnorm} as follows
\begin{equation}\label{ddotnorm_Sob}
\|\uld{\bar{r}}\|_{\YYY^{mod}}
%\leq C\, \Bigl(\frac{\beta}{\tau}+1\Bigr) \|\uld{\bar{r}}\|_{H^1(0,T;L^2(\Omega))}.
\leq C_0\, 
\Bigl(\frac{(\alpha/\tau)}{1-e^{-2(\alpha/\tau)T_0}}\,
\Bigl(1+\left(\frac{\tau}{\beta}\right)^{\check{\orti}}\Bigr)\,e^{2(\alpha/\tau)\,(T-T_0)}
+1\Bigr)^{1/2}
\|\uld{\bar{r}}\|_{H^{\check{\orti}}(0,T;H^{s-\check{\orti}}(\Omega))}
\end{equation}
with some constant $C_0>0$.

In the setting \eqref{setting_s_orti}, the \textbf{Taylor remainder estimate} \eqref{errTay} can be continued as follows.
\begin{equation*}%\label{errTay}
\begin{aligned}
\text{err}_{Tay}
%\leq\|\uld{\slo}\|_{L^2(\Omega)} \|\tilde{\uvec}-\uvec^0\|_{h^1(L^\infty(\Omega))}
%+\|\tilde{\slo}-\slo^0\|_{L^\infty(\Omega)} \|\uld{\uvec}\|_{h^1(L^2(\Omega))}\\&
%+C(T)\|\uld{\nlc}\|_{L^2(\Omega)} \Bigl(\|\tilde{\uvec}\|_{h^1(L^\infty(\Omega))}+\|\uvec^0\|_{h^1(L^\infty(\Omega))}\Bigr)\|\tilde{\uvec}-\uvec^0\|_{h^1(L^\infty(\Omega))}\\&
%+C(T)\|\nlc\|_{L^\infty(\Omega)} \Bigl(\|\tilde{\uvec}\|_{h^1(L^\infty(\Omega))}+\|\uvec\|_{h^1(L^\infty(\Omega))}\Bigr)
%\|\uld{\uvec}\|_{h^1(L^2(\Omega))},
\leq \tilde{C}(s-\check{\orti},s,\Omega) \Bigl(&
\|\uld{\slo}\|_{H^s(\Omega)} \|\tilde{\uvec}-\uvec^0\|_{h^{\check{\orti}}(W^{\check{s},d/s}(\Omega)\cap L^{d/\check{\orti}}(\Omega))}
\\&+\|\tilde{\slo}-\slo^0\|_{W^{\check{s},d/(s-\check{\orti})}(\Omega)\cap L^\infty(\Omega)} \|\uld{\uvec}\|_{h^{\check{\orti}}(H^{s-\check{\orti}}(\Omega))}
\\&+\|\uld{\nlc}\|_{H^s(\Omega)} \|\tilde{\uvec}^2-(\uvec^0)^2\|_{h^{\check{\orti}}(W^{\check{s},d/s}(\Omega)\cap L^{d/\check{\orti}}(\Omega))}
\\&+\|\nlc(\tilde{\uvec}+\uvec)\|_{(h^{\check{\orti}}\cap\ell^1)(W^{\check{s},d/(s-\check{\orti})}(\Omega)\cap L^\infty(\Omega))} 
\|\uld{\uvec}\|_{h^{\check{\orti}}(H^{s-\check{\orti}}(\Omega))}\Bigr),
\end{aligned}
\end{equation*}
where $\|\vec{v}\|_{\ell^1}=\|v\|_{L^\infty(\Omega)}$ and we have used 
%the fact that $H^1(0,T;L^\infty(\Omega))$ is a Banach algebra
%\[\|v\cdot w\|_{H^1(0,T;L^\infty(\Omega))}\leq C(T)\|v\|_{H^1(0,T;L^\infty(\Omega))}\|w\|_{H^1(0,T;L^\infty(\Omega))}\]
%along with 
the definition of $\mathcal{B}$ as the Fourier coefficient of a product and a Kato Ponce Vega inequality cf., e.g. \cite{KenigPonceVega1993}, with $\frac{1}{2}=\frac{1}{p_1}+\frac{1}{q_1}=\frac{1}{p_2}+\frac{1}{q_2}$,
combined with continuity of the embedding $H^s(\Omega)\to L^p(\Omega)$ 
\begin{equation}\label{eq:KatoPonceVega}
\begin{aligned}
%&\|f\,g\|_{H^s(\Omega)}\leq C(s,\Omega)\|f\|_{H^s(\Omega)}\|g\|_{L^\infty(\Omega)\cap W^{s,p^#}(\Omega)}\,, \\ 
%&f, \, g\in H^s(\Omega)\cap L^\infty(\Omega)\,.
&\|f\,g\|_{H^{\check{s}}(\Omega)}\leq C(\check{s},\Omega)\Bigl(
\|f\|_{L^{p_1}(\Omega)}\|g\|_{W^{\check{s},q_1}(\Omega)}
+\|f\|_{W^{\check{s},p_2}(\Omega)}\|g\|_{L^{q_2}(\Omega)}\Bigr)\\
&\leq \tilde{C}(\check{s},s,\Omega) 
\begin{cases}
\|f\|_{H^s(\Omega)}\|g\|_{W^{\check{s},q_1}(\Omega)\cap L^{q_2}(\Omega)}
\text{ with }s-\frac{d}{2}\geq-\frac{d}{p_1} \text{ and } s-\frac{d}{2}\geq\check{s}-\frac{d}{p_2}
\\
\|g\|_{H^{\check{s}}(\Omega)}\|f\|_{W^{\check{s},p_2}(\Omega)\cap L^\infty(\Omega)}
\text{ with } \check{s}-\frac{d}{2}\geq-\frac{d}{q_2} %q_1=2, p_1=\infty
\end{cases}
\,.
%\\ &f, \, g\in H^s(\Omega)\cap L^\infty(\Omega)\quad 
\end{aligned}
\end{equation}
This leads us to defining $U_c$ by \eqref{Uc}.

\begin{theorem}\label{thm:nlstab}
For $\xi_0=(\slo^0,0,\uvec^0)$ with \eqref{u0separable}, $\psi_\no$ as in Lemma~\ref{lem:A}, $\no\in\{1,2=:\No\}$, and satisfying \eqref{muellbdd},
there exists 
$C_0>0$ such that with $\overline{C}(\tau)$ defined as in \eqref{olCtau_mainresults}
%\begin{equation}\label{olCtau}
%\overline{C}(\tau):=C_0\, 
%\Bigl(\frac{(\alpha/\tau)}{1-e^{-2(\alpha/\tau)T_0}}\,\,e^{2(\alpha/\tau)\,(T-T_0)}
%\Bigl(1+\left(\frac{\tau}{\beta}\right)^{\check{\orti}}\Bigr)+1\Bigr)^{1/2}
%\end{equation}
the stability estimate \eqref{stabest1} holds in the form
\begin{equation}\label{stabest}
\|\tilde{\xi}-\xi\|_{H^s(\Omega)^2\times h^{\check{\orti}}(H^{s-\check{\orti}}(\Omega))^2}\leq \max\{1,\,\overline{C}(\tau)\}\, 
\|\vec{F}(\tilde{\xi})-\vec{F}(\xi)\|_{h^{\check{\orti}}(H^{s-\check{\orti}}(\Omega))^2\times\YYY^{obs}} 
\end{equation}
for all $\tilde{\xi}$, $\xi$ $\in U_{c(\tau)}$, 
$c(\tau)=(2\max\{1,\,\overline{C}(\tau)\})^{-1}$, 
with $\|\cdot\|_{\YYY^{obs}}$ according to \eqref{normYmodYobs}, $s>\frac12$, $\check{\orti}\in[0,\min\{s,1\}]$,
and 
%$\|\cdot\|_{h^1(L^2(\Omega))}$ 
{$\|\cdot\|_{h^{\check{\orti}}(H^{s-\check{\orti}}(\Omega))}$}
according to \eqref{BSnorms}.
\end{theorem}
%In particular in case $\check{\orti}=0$, the estimate is independent of $\tau$ and in fact, it is readily checked that it also holds for the Westervelt in place of the JMGT model.

%Indeed the $\YYY^{obs}$ is much stronger in case $\tau=0$, reflecting the severe ill-posedness of the inverse problem in case $\tau=0$.
The fact that the constants degenerate as $\tau$ vanishes, reflects the severe ill-posedness of the inverse problem in the Westervelt case $\tau=0$.
\\
We will thus now study regularization of the case of $\tau=0$ (or $\tau=\tau_0$ small)  by approaching it as the limit for decreasing $\tau$.

\section{JMGT as a quasi reversibility model for Westervelt: proof of Theorem~\ref{thm:quasireversibility_mainresults}}
\label{sec:quasireversibility}

The key task for establishing a parameter dependent family of reconstructions $(\xi_\tau)_{\tau>\tau_0}$ as a regularization method is, given any family of noisy data $(\obs^\delta)_{\delta>0}$ with noise level $\delta$
\begin{equation}\label{delta}
\|\obs_\no^\delta-\obs_\no\|_{\widetilde{\YYY}^{obs}}\leq\delta\, \quad \no\in\{1,2\}
\end{equation}
to prove that there exists a choice $\tau=\tau(\delta)$ such that 
\[
\xi_{\tau(\delta)}\to \xi_{\tau_0} \quad \text{ as }\delta\to0.
\]
Here the norm in $\widetilde{\YYY}^{obs}$ should be realistic (unlike the norm in $\YYY^{obs}$ used in the uniqueness proof), so that \eqref{delta} can be achieved by possible post-smoothing of the observed noisy data. We therefore strive to work with a Bochner-Sobolev norm as much as possible, up to a part that cannot be quantified (because it relies on unique continuation). We will comment on how to tackle the latter in Remark~\ref{rem:datasmoothing} below.

To this end, we have to track $\tau$ dependence in the analysis above. This obviously affects $\vec{F}=\vec{F}_\tau$ through the leading order in time term in $\vec{F}_\tau^{mod}$, but also the radius $c=c(\tau)$ and the Lipschitz constant in \eqref{stabest} and, via the location of the poles, also the $\YYY$ norm.
The latter can be dealt with by substituting $\YYY=\YYY_\tau$ by a space equipped with a $\tau$ independent norm 
$\widetilde{\YYY}=\widetilde{\YYY}^{mod}\times\widetilde{\YYY}^{obs}$, with 
\begin{equation}\label{Ytil}
\widetilde{\YYY}^{mod}=h^{\check{\orti}}(H^{s-\check{\orti}}(\Omega))^2 \, \quad
\widetilde{\YYY}^{obs}=\{\uld{\bar{p}}\, : \, \Tinv\uld{\bar{p}}\in W^{1,1}(0,T;H^{s+1}(\Omega))\}
\end{equation}
cf., \eqref{stabest} above and \eqref{ddotnorm_Sob_obs} below, for which we will show (in the proof of Theorem~\ref{thm:quasireversibility}) that 
\begin{equation}\label{Ytilobs_est}
\|v\|_{\YYY^{obs}}\leq \widetilde{C}(\tau)\|v\|_{\widetilde{\YYY}^{obs}}\quad v\in \widetilde{\YYY}^{obs}\subseteq\YYY. 
\end{equation}
holds with 
\begin{equation}\label{Ctiltau}
\begin{aligned}
&\|\uld{\bar{\obs}}\|_{\widetilde{\YYY}^{obs}}:=\|\Tinv\uld{\bar{\obs}}\|_{W^{1,1}(0,T;H^{s+1}(\Omega))}, \\
&\widetilde{C}(\tau)=
C_1\,\left(\frac{(\alpha/\tau)}{1-e^{-2(\alpha/\tau)T_0}}\,\,e^{2(\alpha/\tau)\,(T-T_0)}
\left(\left(\frac{\beta}{\tau}\right)^{\check{\orti}}+1\right)\right)^{1/2}
\end{aligned}
\end{equation}
for some $C_1$ independent of $\tau$.

The preimage space will remain
\begin{equation}\label{XXX}
\XXX=H^s(\Omega)^2\times h^{\check{\orti}}(H^{s-\check{\orti}}(\Omega))^2
\end{equation}
with $s>\frac12$, $\check{\orti}\in[0,\min\{s,1\}]$, cf. \eqref{setting_s_orti}.
This can be easily seen to yield, in place of \eqref{stabest}, the stability estimate 
%Theorem~\ref{thm:nlstab},
\begin{equation}\label{stabest2}
\|\tilde{\xi}-\xi\|_{H^s(\Omega)^2\times h^{\check{\orti}}(H^{s-\check{\orti}}(\Omega))^2}\leq \max\{\overline{C}(\tau),\,\widetilde{C}(\tau)\}\, 
\|\vec{F}(\tilde{\xi})-\vec{F}(\xi)\|_{h^{\check{\orti}}(H^{s-\check{\orti}}(\Omega))^2\times\widetilde{\YYY}^{obs}}. 
\end{equation}

While we aim to find the solution $\xi_{\tau_0}$ to the inverse problem with noiseless data
\begin{equation}\label{Ftau0y}
\vec{F}_{\tau_0}(\xi_{\tau_0})=\vec{y}
\end{equation}
we denote by $\xi_\tau^\delta=(\slo_\tau,\nlc_\tau,\uvec_\tau)$ the solution to the regularized noisy data problem  
\begin{equation}\label{Ftauydelta}
\vec{F}_\tau(\xi_\tau^\delta)=\vec{y}^\delta 
\end{equation}
with $\vec{y}=(\vec{r}_1,\vec{r}_2,\vec{\obs}_1,\vec{\obs}_2)$, $\vec{y}=(\vec{r}_1,\vec{r}_2,\vec{\obs}_1^\delta,\vec{\obs}_2^\delta)$. 
Using the local stability result 
%Theorem~\ref{thm:nlstab}, 
\eqref{stabest2}, as well as \eqref{Ftau0y} and \eqref{Ftauydelta}, we have 
\begin{equation}\label{totalerror}
\begin{aligned}
&\|\xi_\tau^\delta-\xi_{\tau_0}\|_{\XXX}
\leq \max\{\overline{C}(\tau),\,\widetilde{C}(\tau)\}\,\|\vec{F}_\tau(\xi_\tau^\delta)-\vec{F}_\tau(\xi_{\tau_0})\|_{\widetilde{\YYY}} \\
&= \max\{\overline{C}(\tau),\,\widetilde{C}(\tau)\}\,\|\vec{y}^\delta-\vec{y}+\vec{F}_{\tau_0}(\xi_{\tau_0})-\vec{F}_\tau(\xi_{\tau_0})\|_{\widetilde{\YYY}}\\ 
&\leq \max\{\overline{C}(\tau),\,\widetilde{C}(\tau)\}\,\Bigl(\delta+(\tau-\tau_0)\|(\partial_t u_{\no,\tau_0})_{\no=1}^2\|_{H^{\check{\orti}}(0,T;H^{s-\check{\orti}}(\Omega))^2}\Bigr), 
\end{aligned}
\end{equation}
provided $\xi_\tau^\delta$, $\xi_{\tau_0}$ $\in U_{c(\tau)}$, $c(\tau)=(2\max\{\overline{C}(\tau),\,\widetilde{C}(\tau)\})^{-1}$.

Note that in case $\tau_0=0$ an essential necessary condition for convergence is $\max\{\overline{C}(\tau),\,\widetilde{C}(\tau)\}\,\tau\to0$ as $\tau\to0$, that is, $T_0=T$ and $\frac{\check{\orti}+1}{2}<1$.

This immediately implies the following result.
\begin{theorem}\label{thm:quasireversibility}
Let $\partial_t u_{\no,\tau_0}\in H^{\check{\orti}}(0,T;H^{s-\check{\orti}}(\Omega))^2$,
$\no\in\{1,2\}$, 
with $s>\frac12$, $\check{\orti}\in[0,\min\{s,1\}]$, cf. \eqref{setting_s_orti},
 and either \\
(a) $\tau_0>0$ or \\
(b) $\tau_0=0$ and $T_0=T$ and $\check{\orti}<1$ \\
hold.
Then with a choice $\tau=\tau(\delta)$ such that 
\[
\tau(\delta)\to\tau_0, \quad \max\{\overline{C}(\tau(\delta)),\,\widetilde{C}(\tau(\delta))\}\,\delta\to0, \quad 
\text{ as }\delta\to0
\]
we have convergence $\|\xi_{\tau(\delta)}^\delta-\xi_{\tau_0}\|_{\XXX}\to0$ as $\delta\to0$.
\end{theorem}

\begin{proof}
The remainder of the proof is concerned with deriving estimate \eqref{Ytilobs_est} with \eqref{Ctiltau} of the $\YYY^{obs}$ norm by a $\tau$ independent Sobolev type norm.
To this end, we have to be aware of the difficulty that the action of $(\text{Tr}_\Sigma^{{s,\ell}})^{-1}$ in dependence of $\ell$, cannot be quantified in terms of Sobolev norms in general. 
A useful property of $(\text{Tr}_\Sigma^{{s,\ell}})^{-1}$ is its independence of $\tau$, though.
Therefore, we exploit the interchangability of spatial trace and projection with temporal differential operators (such as taking the residue in frequency domain), so that with \eqref{resintrace}, $(\text{Tr}_\Sigma^{{s,\ell}})^{-1}$ to be understood as applied component wise, and $\text{Proj}_{\text{tr}_\Sigma\mathbb{E}^\ell}$ the projection in $H^{s-1/2}(\Sigma)$
\[
\begin{aligned}
&(\text{Tr}_\Sigma^{{s,\ell}})^{-1}\,\left[\widetilde{\mathfrak{M}}({\pole_\ell})^{-1}
\, \text{res}(\widetilde{\uld{\bar{\obs}}};{\pole_\ell})\right]
=\widetilde{\mathfrak{M}}({\pole_\ell})^{-1}
\, (\text{Tr}_\Sigma^{{s,\ell}})^{-1}\, \text{Proj}_{\text{tr}_\Sigma\mathbb{E}^\ell} \text{res}(\widetilde{\uld{\bar{\obs}}};{\pole_\ell})\\
&=\widetilde{\mathfrak{M}}({\pole_\ell})^{-1}\, (\text{Tr}_\Sigma^{{s,\ell}})^{-1}\, \text{res}(\text{Proj}_{\text{tr}_\Sigma\mathbb{E}^\ell}\widetilde{\uld{\bar{\obs}}};{\pole_\ell})
=\widetilde{\mathfrak{M}}({\pole_\ell})^{-1}\, \text{res}(\Tinv^\ell\widetilde{\uld{\bar{\obs}}};{\pole_\ell})
\end{aligned}
\]
with 
\begin{equation}\label{Tinv_ell}
\Tinv^\ell\,v:=\sum_{k\in K^\ell} \Bigl((\text{Tr}_\Sigma^{{s,\ell}})^{-1}
\text{Proj}_{\text{tr}_\Sigma\mathbb{E}^\ell} \,v\Bigr)_k \varphi^{\ell,k}\ \in \mathbb{E}^\ell
v\in H^{s-1/2}(\partial\Omega)
\end{equation}
%\textcolor{red}{While the $\mathbb{E}^\ell$ are mutually orthogonal, the linear spaces $\text{Proj}_{\text{tr}_\Sigma\mathbb{E}^\ell}$ may have nontrivial overlaps. Can this cause trouble with well-definedness of the operators $\Tinv^\ell$?}
and make use of analogy of the two parts of the $\YYY$ norm in \eqref{normYmodYobs}. 
To this end, we estimate the interpolant according to \eqref{interpol_p},
abbreviating $f=\Tinv\uld{\bar{\obs}}$, 
\[
\begin{aligned}
%&\text{res}(\Tinv^\ell\widetilde{\uld{\bar{\obs}}};{\pole_\ell})
&\text{res}(\widetilde{f};{\pole_\ell})
=\lim_{o\to p_\ell} \tfrac{2}{T}\int_0^T(o-p_\ell)\, f(t) e^{-ot}\,dt\\
&=\lim_{o\to p_\ell} \tfrac{2}{T}\Bigl(\int_0^T (f_t(t)-\pole_\ell f(t)) e^{-ot}\,dt-\left[f(t)e^{-ot}\right]_0^T\Bigr)\\
&=\lim_{o\to p_\ell} \tfrac{2}{T}\Bigl(\int_0^T (f_t(t)-\pole_\ell f(t)) e^{-ot}\,dt-f(0)\bigl(e^{-oT}-1\bigr)\Bigr),
\end{aligned}
\]
(where in the last identity we have used $T$ periodicity of $f$, which is not essential for this estimate, though).
\\
Hence, for any $\bar{s}\geq0$
\[
\begin{aligned}
&\|\text{res}(\Tinv^\ell\widetilde{\uld{\bar{\obs}}};{\pole_\ell})\|_{H^{\bar{s}}(\Omega)}\\
&\leq \lim_{o\to p_\ell} \tfrac{2}{T}\Bigl(\|\Tinv^\ell\uld{\bar{\obs}}_t\|_{L^1(0,T;H^{\bar{s}}(\Omega))}\|e^{-o\cdot}\|_{L^\infty(0,T)}+|\pole_\ell| \|\Tinv^\ell\uld{\bar{\obs}}\|_{L^\infty(0,T;H^{\bar{s}}(\Omega))}\|e^{-o\cdot}\|_{L^1(0,T)}\\
&\hspace*{5cm} + \|\Tinv^\ell\uld{\bar{\obs}}\|_{L^\infty(0,T;H^{\bar{s}}(\Omega))}|e^{-oT}-1|\Bigr)\\
&\leq \tfrac{2}{T}e^{-\Re(p_\ell)T}
\Bigl(1+\bigl(|1-e^{p_\ell T}|+|\pole_\ell|\, \sup_{x>0}\frac{1-e^{-x}}{x}\bigr)C_{W^{1,1},L^\infty}^{(0,T)}\Bigr)
\|\Tinv^\ell\uld{\bar{\obs}}\|_{W^{1,1}(0,T;H^{\bar{s}}(\Omega))}.
%\\&\leq C\,\tfrac{2}{T}\,e^{-\Re(p_\ell)T}\, \sqrt{\frac{\beta}{\tau}} \,\left(1+ \frac{C\alpha}{\lambda_\ell}\right)\|\Tinv^\ell\widetilde{\uld{\bar{\obs}}}\|_{W^{1,1}(0,T;H^{\bar{s}+1}(\Omega))},
\end{aligned}
\]
%where we have used the pole asymptotics \eqref{poles}.
Moreover, we have
\[
\begin{aligned}
&\frac{\Theta(o)\Psi'(o)}{o^2}
=\frac{-\vartheta'(o)\Theta(o)+\vartheta(o)\Theta'(o)}{\Theta(o)o^2}
=\frac{-(3\tau o^2+2\slo^0 o)(\beta o +1)+(\tau o^3+\slo^0 o^2)\beta}{(\beta o +1)o^2}\\
&=\frac{-2\tau -\slo^0/o -3\tau/(\beta o)-2\slo^0/(\beta o^2)}{1 +1/(\beta o)}
\end{aligned}
\]
hence, due to $|\pole_\ell|\sim\sqrt{\frac{\beta}{\tau}\lambda_\ell}$, cf. \eqref{poles}
\[
\begin{aligned}
&\left|\frac{\Theta({\pole_\ell})\Psi'({\pole_\ell})}{\pole_\ell^2}\right|
\|\text{res}(\Tinv^\ell\widetilde{\uld{\bar{\obs}}};{\pole_\ell})\|_{H^{\bar{s}}(\Omega)}\\
&\leq C \,e^{-\Re(p_\ell)T}\, (\tau|\pole_\ell|+1)\,\|\Tinv^\ell\uld{\bar{\obs}}\|_{W^{1,1}(0,T;H^{\bar{s}}(\Omega))}\\
&\leq \check{\check{C}} \,e^{-\Re(p_\ell)T}\, 
\Bigl(\sqrt{\beta\tau} 
\,\|\Tinv^\ell\uld{\bar{\obs}}\|_{W^{1,1}(0,T;H^{\bar{s}+1}(\Omega))}
+\|\Tinv^\ell\uld{\bar{\obs}}\|_{W^{1,1}(0,T;H^{\bar{s}}(\Omega))}
\Bigr).
\end{aligned}
\]
Setting $\bar{s}=s$, the analogon to \eqref{lemAnormpole} is therefore
\[
\begin{aligned}
&\lambda_\ell^{s}\, \sum_{k\in K^j}\left|\frac{\Theta({\pole_\ell})\Psi'({\pole_\ell})}{\pole_\ell^2}\,
(\text{Tr}_\Sigma^{{s,\ell}})^{-1}\,\left[\widetilde{\mathfrak{M}}({\pole_\ell})^{-1}
\, \text{res}(\widetilde{\uld{\bar{\obs}}};{\pole_\ell})\right]\right|^2\\
&\leq|\widetilde{\mathfrak{M}}({\pole_\ell})^{-1}|^2\,\left|\frac{\Theta({\pole_\ell})\Psi'({\pole_\ell})}{\pole_\ell^2}\right|^2
\|\text{res}(\Tinv^\ell\widetilde{\uld{\bar{\obs}}};{\pole_\ell})\|_{H^{s}(\Omega)}^2\\
&\leq \check{\check{C}} \, C_\mu^2 \, 
\frac{-2\Re(\pole_\ell)}{1-e^{2\Re(\pole_\ell)\,T_0}} \,e^{-2\Re(\pole_\ell)\,(T-T_0)}\,\\
&\qquad\qquad\Bigl(\beta\tau \,\|\Tinv^\ell\uld{\bar{\obs}}\|_{W^{1,1}(0,T;H^{s+1}(\Omega))}^2
+\|\Tinv^\ell\uld{\bar{\obs}}\|_{W^{1,1}(0,T;H^{s}(\Omega))}^2\Bigr)
.
\end{aligned}
\]
With 
%$\Tinv\widetilde{\uld{\bar{\obs}}}:=\sum_{\ell\in\mathbb{N}} \Tinv^\ell\widetilde{\uld{\bar{\obs}}}$, 
\begin{equation}\label{Tinv}
\Tinv v:=\sum_{\ell\in\mathbb{N}} \Tinv^\ell v, \quad 
v\in H^{s-1/2}(\Sigma), 
\end{equation}
(note that $v\in H^{s-1/2}(\Sigma)$ implies $\text{Proj}_{\text{tr}_\Sigma\mathbb{E}^\ell} v\in H^{s-1/2}(\Sigma))$, $\ell\in\mathbb{N}$,)
this yields, analogously to \eqref{BSnormest}, \eqref{ddotnorm_Sob}, 
\begin{equation}\label{ddotnorm_Sob_obs}
\begin{aligned}
\|\uld{\bar{p}}\|_{\YYY^{obs}}^2 
%\leq& \check{\check{C}} \, C_\mu^2 \, 
%\sum_{\ell\in\mathbb{N}} \frac{-2\Re(\pole_\ell)}{1-e^{2\Re(\pole_\ell)\,T_0}} \,e^{-2\Re(\pole_\ell)\,(T-T_0)} \, 
% \Bigl(\sum_{m\in\mathbb{N}} 
%\frac{|o_m|^{4+2\check{\orti}}\lambda_\ell^{-(s-\check{s})}}{|\vartheta(o_m)+\Theta(o_m)\lambda_\ell|^2} |\mathfrak{M}_m|^2 + 1\Bigr)\\
%&\qquad\qquad\Bigl(\beta\tau \,\|\Tinv^\ell\uld{\bar{\obs}}\|_{W^{1,1}(0,T;H^{s+1}(\Omega))}^2
%+\|\Tinv^\ell\uld{\bar{\obs}}\|_{W^{1,1}(0,T;H^{s}(\Omega))}^2\Bigr)
%\\
\leq& \check{\check{C}} \, C_\mu^2 \frac{(\alpha/\tau)}{1-e^{-2(\alpha/\tau)T_0}}\,\,e^{2(\alpha/\tau)\,(T-T_0)}
\left(C_\psi\Chat_{\check{\orti}}\left(\frac{\beta}{\tau}\right)^{\check{\orti}}+1\right)\,\\
&\Bigl(\beta\tau \,\|\Tinv\uld{\bar{\obs}}\|_{W^{1,1}(0,T;H^{s+1}(\Omega))}^2
+\|\Tinv\uld{\bar{\obs}}\|_{W^{1,1}(0,T;H^{s}(\Omega))}^2\Bigr)
\end{aligned}
\end{equation}
provided $s-\check{s}\geq\check{\orti}$, that is, \eqref{Ytilobs_est} with 
%\begin{equation*}%\label{Ctiltau}
%\begin{aligned}
%&\|\uld{\bar{\obs}}\|_{\widetilde{\YYY}^{obs}}:=\|\Tinv\uld{\bar{\obs}}\|_{W^{1,1}(0,T;H^{s+1}(\Omega))}, \\
%&\widetilde{C}(\tau)=
%C_1\,\left(\frac{(\alpha/\tau)}{1-e^{-2(\alpha/\tau)T_0}}\,\,e^{2(\alpha/\tau)\,(T-T_0)}
%\left(\left(\frac{\beta}{\tau}\right)^{\check{\orti}}+1\right)\right)^{1/2},
%\end{aligned}
%\end{equation*}
%cf. 
\eqref{Ctiltau}.
%Since the $\widetilde{\YYY}^{obs}$ component does not appear in the Taylor remainder, we have no such constraint as \eqref{cond_sorti_tay} here.
\end{proof}

\begin{remark}[data smoothing]\label{rem:datasmoothing}
To obtain \eqref{delta} from noisy observations that typically only allow for a noise estimate in a low order norm (since only values but not derivative values can be measured) some data smoothing needs to be employed.
Lifting an $L^2(0,T;L^2(\partial\Omega))$ (or $W^{1,1}(0,T;H^s(\Omega))$) noise level to a typically larger 
$W^{1,1}(0,T;H^{s-1/2}(\partial\Omega))$ (or $W^{1,1}(0,T;H^{s+1}(\Omega))$) noise level by appropriate filtering is a standard task, which we do not discuss here; see, e.g., \cite{KlannRamlau2008}.
Likewise, application of $\Tinv$, that is, stable extension of boundary data to the interior has been discussed in many references see, e.g., \cite[Ch. 5, 14]{Hanke:2017} 
%https://epubs.siam.org/doi/10.1137/050626053
%https://epubs.siam.org/doi/abs/10.1137/23M1622064
and is not specific to the multiparameter US tomography problem we are focusing on here. Therefore, we only briefly sketch a simple approach for applying $\Tinv$, that can be used if an eigensystem of the Laplacian is available (e.g., on simple geometries such as rectangular ones, as relevant for simulating experiments in water tanks or ellipsoidal ones, as relevant for phantoms).

For simplicity of exposition we skip the subscript $\no$ (the measurement number) as well as time dependence, since the (inverse) trace operator is anyway just to be applied point wise in time. Our taks is therefore, given an approximation $\tilde{p}\in H^{s-1/2}(\partial\Omega)$ of a sufficiently regular exact observation $p\in H^{S-1/2}(\partial\Omega)$, $S>s$, satisfying $\|\tilde{p}-p\|_{H^{s-1/2}(\partial\Omega)}\leq\tilde{\delta}$, to construct $\tilde{v}\in H^s(\Omega)$ such that $\|\tilde{v}-\Tinv p\|_{H^s(\Omega)}\leq \delta$ where $\delta=\delta(\tilde{\delta})\to0$ as $\tilde{\delta}\to0$. 
We will do so by applying the so-called least squares version of regularization by discretization, cf. e.g. \cite{EnglHankeNeubauer:1996,GroetschNeubauer:1988,Kirsch2021,VainikkoHaemarik:1985,projBanach} and the references therein, by restricting the pre-image space to the finite dimensional subspace
\[
V_L:=\text{span}\{\varphi^{\ell,k}\,:\, k\in K^\ell,\ \ell\in\{1,\ldots,L\}\}
\]
and defining 
\[
v_L^{\tilde{\delta}}\in\text{argmin}_{v\in V_L}\|\text{tr}_\Sigma v-\tilde{p}\|_{H^{s-1/2}(\partial\Omega)}
\]
for some properly chosen $L\in\mathbb{N}$.
To this end, observe that $\Tinv$ is a right inverse of $\text{tr}_\Sigma$ since
\[
\text{tr}_\Sigma\, \Tinv\, w = \text{tr}_\Sigma \sum_{\ell\in\mathbb{N}} \sum_{k\in K^\ell} \Bigl((\text{Tr}_\Sigma^{{s,\ell}})^{-1}\text{Proj}_{\text{tr}_\Sigma\mathbb{E}^\ell}\, w \Bigr)_k \varphi^{\ell,k}=w, \quad w\in H^{s-1/2}(\partial\Omega)
\]
and on $V_L$ it is in fact the inverse, since 
\[
\Tinv\, \text{tr}_\Sigma\, v = v, \qquad v\in V_L
\]
by definition of $\Tinv$.
Defining 
\[
W_L:= \text{tr}_\Sigma\,(V_L), \qquad 
\kappa_L:=\max_{\phi\in V_L}\frac{\|\phi\|_{H^s(\Omega)}}{\|\text{tr}_\Sigma\phi\|_{H^{s-1/2}(\partial\Omega)}}
%=\max_{\phi\in V_L,\,\|\phi\|_{H^s(\Omega)}=1}\frac{1}{\|\text{tr}_\Sigma\phi\|_{H^{s-1/2}(\partial\Omega)}} 
\]
%minimizer of continuous function on compact set
we can estimate
\[
\begin{aligned}
\|v_L^{\tilde{\delta}}-\Tinv p\|_{H^s(\Omega)}^2
&=\|(I-\text{Proj}_{V_L})\Tinv p\|_{H^s(\Omega)}^2 + \|v_L^{\tilde{\delta}}-\text{Proj}_{V_L}\Tinv p\|_{H^s(\Omega)}^2\\
&\leq\|(I-\text{Proj}_{V_L})\Tinv p\|_{H^s(\Omega)}^2 + \kappa_L^2\|\text{tr}_\Sigma v_L^{\tilde{\delta}}-\text{tr}_\Sigma \text{Proj}_{V_L}\Tinv p\|_{H^{s-1/2}(\partial\Omega)}^2
\end{aligned}
\]
where $\text{tr}_\Sigma \text{Proj}_{V_L}\Tinv p= \text{Proj}_{W_L}p$ and $e^V_L:=\|(I-\text{Proj}_{V_L})\Tinv p\|_{H^s(\Omega)}\to0$ as $L\to\infty$.
Using the triangle inequality and minimality of $v_L^{\tilde{\delta}}$, we can further estimate  
\[
\begin{aligned}
&\|\text{tr}_\Sigma v_L^{\tilde{\delta}}-\text{Proj}_{W_L}p\|_{H^{s-1/2}(\partial\Omega)}
\leq\|\text{tr}_\Sigma v_L^{\tilde{\delta}}-\tilde{p}\|_{H^{s-1/2}(\partial\Omega)}
+\|\text{Proj}_{W_L}p-\tilde{p}\|_{H^{s-1/2}(\partial\Omega)}\\
&\leq 2\|\text{Proj}_{W_L}p-\tilde{p}\|_{H^{s-1/2}(\partial\Omega)}
\leq 2(\|(I-\text{Proj}_{W_L})p\|_{H^{s-1/2}(\partial\Omega)}+{\tilde{\delta}})
\end{aligned}
\]
For $p\in H^{S-1/2}(\partial\Omega)$, with sufficiently large regularity offset $S-s>0$, the projection error $\|(I-\text{Proj}_{W_L})p\|_{H^{s-1/2}(\partial\Omega)}$ tends to zero fast enough so that (in spite of the growth of $\kappa_L$) also $e^W_L:=\kappa_L\|(I-\text{Proj}_{W_L})p\|_{H^{s-1/2}(\partial\Omega)}\to0$ as $L\to\infty$. Thus we have 
\[
\begin{aligned}
\|v_L^{\tilde{\delta}}-\Tinv p\|_{H^s(\Omega)}
\leq {e^V_L} + 2 {e^W_L} + 2\kappa_L{\tilde{\delta}}
\end{aligned}
\]
and therefore a choice of $L=L({\tilde{\delta}})$ such that $L({\tilde{\delta}})\to\infty$ and $\kappa_{L({\tilde{\delta}})}{\tilde{\delta}}\to0$ as ${\tilde{\delta}}\to0$ -- ideally by balancing $e^V_{L({\tilde{\delta}})} + e^W_{L({\tilde{\delta}})}\sim \kappa_{L({\tilde{\delta}})}{\tilde{\delta}}$ -- we obtain the desired result with 
$\delta({\tilde{\delta}}):={e^V_{L({\tilde{\delta}})}} + 2 {e^W_{L({\tilde{\delta}})}} + 2\kappa_{L({\tilde{\delta}})}{\tilde{\delta}}\to0$ as ${\tilde{\delta}}\to0$.
\end{remark}

\section*{Acknowledgment}
This research was funded in part by the Austrian Science Fund (FWF) 
[10.55776/P36318]. 
For open access purposes, the author has applied a CC BY public copyright license to any author accepted manuscript version arising from this submission.
%The author wishes to thank the reviewers for their careful reading of the manuscript and their detailed reports with valuable comments and suggestions that have led to an improved version of the paper.

%\bibliographystyle{siamplain}
%\bibliography{lit,refs_JMGT-Neumann}

\end{document}